\documentclass[reqno,11pt,draft]{amsart}
\usepackage{pgf,tikz}
\usepackage{amsmath}
\usepackage{amsthm}
\usepackage{amssymb}

\theoremstyle{plain}
\newtheorem{thm}{Theorem}[section]
\newtheorem{lem}[thm]{Lemma}
\newtheorem{cor}[thm]{Corollary}
\newtheorem{prop}[thm]{Proposition}

\theoremstyle{remark}
\newtheorem*{rem}{Remark}
\newtheorem*{ex}{Example}

\theoremstyle{definition}
\newtheorem*{defn}{Definition}

\newtheorem{cosa}[thm]{}

\numberwithin{equation}{thm}

\newcommand{\CCD}{\mathcal{D}}
\newcommand{\CE}{\mathcal{E}}
\newcommand{\CF}{\mathcal{F}}

\newcommand{\CM}{\mathcal{M}}

\newcommand{\CO}{\mathcal{O}}

\newcommand{\CS}{\mathcal{S}}

\newcommand{\CV}{\mathcal{V}}
\newcommand{\CU}{\mathcal{U}}
\newcommand{\CW}{\mathcal{W}}

\newcommand{\CY}{\mathcal{Y}}

\newcommand{\CCC}{\boldsymbol{\mathsf{C}}}
\newcommand{\D}{\boldsymbol{\mathsf{D}}}
\newcommand{\K}{\boldsymbol{\mathsf{K}}}
\newcommand{\LL}{\boldsymbol{\mathsf{L}}}
\newcommand{\R}{\boldsymbol{\mathsf{R}}}
\newcommand{\T}{\boldsymbol{\mathsf{T}}}

\newcommand{\bb}{\mathsf{b}}

\newcommand{\CP}{\mathsf{P}}

\newcommand{\tf}{\mathsf{fg}}

\newcommand{\NN}{\mathbb{N}}
\newcommand{\ZZ}{\mathbb{Z}}
\newcommand{\CPX}{\mathbb{C}}

\newcommand{\ia}{{\mathfrak a}}
\newcommand{\ib}{{\mathfrak b}}

\newcommand{\ip}{{\mathfrak p}}
\newcommand{\iq}{{\mathfrak q}}
\newcommand{\im}{{\mathfrak m}}
\newcommand{\idealn}{{\mathfrak n}}

\newcommand{\holim}[1]{\begin{array}[t]{c} {\rm holim}\\[-7.5 pt]
 {\longrightarrow} \\[-7.5 pt] {\scriptstyle {#1}} \end{array}}
\newcommand{\dirlim}[1]{\underset{\scriptstyle {#1}}{\underset{\longrightarrow}{\lim}}\,}

\newcommand{\lto}{\longrightarrow}
\newcommand{\xto}{\xrightarrow}
\newcommand{\ot}{\leftarrow}

\newcommand{\inc}{\hookrightarrow}
\DeclareMathOperator{\iso}{\tilde{\to}}
\DeclareMathOperator{\osi}{\tilde{\ot}}

\newcommand{\imp}{\Rightarrow}
\newcommand{\dimp}{\Leftrightarrow}

\DeclareMathOperator{\Hom}{Hom}

\DeclareMathOperator{\ext}{Ext}

\DeclareMathOperator{\rhom}{\R{}Hom}

\DeclareMathOperator{\tor}{Tor}

\DeclareMathOperator{\hhtt}{ht}

\DeclareMathOperator{\lth}{L}

\DeclareMathOperator{\Img}{Im}

\DeclareMathOperator{\Ker}{Ker}
\DeclareMathOperator{\spec}{Spec}

\DeclareMathOperator{\supp}{Supp}

\DeclareMathOperator{\h}{H}
\DeclareMathOperator{\id}{id}

\DeclareMathOperator{\Ais}{\mathsf{Ais}}
\DeclareMathOperator{\Aisc}{\mathsf{Ais_{cp}}}
\DeclareMathOperator{\Filsup}{\mathsf{Fil_{sp}}}

\newcommand{\Mod}{\mathsf{Mod}}

\newcommand{\Dc}{\D_{\tf}}
\newcommand{\Dbc}{\D_{\tf}^{\bb}}
\newcommand{\Db}{\D^{\bb}}

\newcommand{\aaa}{\mathsf{a}}
\newcommand{\fff}{\mathsf{f}}
\newcommand{\fcm}{\phi_\textsc{cm}}

\DeclareMathOperator{\ass}{Ass}

\DeclareMathOperator{\Min}{Min}
\DeclareMathOperator{\rad}{rad}
\DeclareMathOperator{\V}{V}

\newcommand{\rhomdot}{\R{}\Hom^\cdot}

\newcommand{\x}[1]{#1^{\times}}
\newcommand{\Dgeq}[1]{\D^{\geq #1}}
\newcommand{\Dleq}[1]{\D^{\leq #1}}
\newcommand{\Dg}[1]{\D^{> #1}}
\newcommand{\Dl}[1]{\D^{< #1}}
\newcommand{\Dcacb}[2]{\D^{[#1,#2]}}

\newcommand{\Dmenos}[1]{\D^{-}_{#1}}
\newcommand{\Dmas}[1]{\D^{+}_{#1}}
\newcommand{\Dcmenos}{\Dmenos{\tf}}
\newcommand{\Dcmas}{\Dmas{\tf}}

\newcommand{\Dcleq}[1]{\D_{\tf}^{\leq #1}}

\newcommand{\aisle}[1]{\mathsf{aisle}\langle #1 \rangle}
\newcommand{\cu}[2]{\CU^{#1}_{#2}}

\newcommand{\tgeq}[2]{\tau^{\geq #1}_{#2}}
\newcommand{\tg}[2]{\tau^{> #1}_{#2}}
\newcommand{\tleq}[2]{\tau^{\leq #1}_{#2}}
\newcommand{\tl}[2]{\tau^{< #1}_{#2}}

\newcommand{\ddd}{\mathsf{d}}

\newcommand{\Wdot }{W}
\newcommand{\DDD}{D}
\newcommand{\Xdot}{X}
\newcommand{\Ydot}{Y}
\newcommand{\Mdot}{M}
\newcommand{\Ndot}{N}
\newcommand{\Bdot}{B}

\newcommand{\Tdot}{T}

\newcommand{\Fdot}{F}
\newcommand{\Edot}{E}

\newcommand{\ie}{\text{\it i.e.}}
\newcommand{\cfr}{\text{\it cf.}}
\newcommand{\lc}{\text{\it loc. cit.}}

\begin{document}

\title[Compactly generated $t$-structures]{Compactly generated $t$-structures on the derived category of a Noetherian ring}

\author[L. Alonso]{Leovigildo Alonso Tarr\'{\i}o}
\address[L. A. T.]{Departamento de \'Alxebra\\
Facultade de Matem\'a\-ticas\\
Universidade de Santiago de Compostela\\
E-15782 Santiago de Compostela, SPAIN}
\email{leoalonso@usc.es}

\author[A. Jerem\'{\i}as]{Ana Jerem\'{\i}as L\'opez}
\address[A. J. L.]{Departamento de \'Alxebra\\
Facultade de Matem\'a\-ticas\\
Universidade de Santiago de Compostela\\
E-15782 Santiago de Compostela, SPAIN}
\email{jeremias@usc.es}

\author[M. Saor\'{\i}n]{Manuel Saor\'{\i}n}
\address[M. S. C.]{Departamento de Matem\'a\-ticas\\
Universidad de Murcia, Aptdo. 4021\\
E-30100 Espinardo, Murcia\\, SPAIN}
\email{msaorinc@um.es}


\date{\today}
\thanks{L.A.T. and A.J.L. have been partially supported by Spain's MEC and E.U.'s FEDER research projects MTM2005-05754 and MTM2008-03465 together with Xunta de Galicia's grant PGIDIT06PXIC207056PN and acknowledge hospitality and support from Purdue University and Universidad de Murcia.\\
\indent M.S. has been
supported by the D.G.I. of the Spanish Ministry of Education and
the Fundaci\'on ``S\'eneca'' of Murcia, with a part of FEDER funds from
the European Union.\\
\indent \emph{Manuel Saor{\'\i}n dedica este art{\'\i}culo a sus padres en el 50 aniversario de su boda}\!\!
}

\begin{abstract}
We study $t$-structures on $\D(R)$ the derived category of modules over a commutative Noetherian ring $R$ generated by complexes in $\Dcmenos(R)$. We prove that they are exactly the compactly generated $t$-structures on $\D(R)$ and describe them in terms of decreasing filtrations by supports of $\spec(R)$. A decreasing filtration by supports $\phi\colon \ZZ \to \spec(R)$ satisfies the weak Cousin condition if for any integer $i$, the set $\phi(i)$ contains all the immediate generalizations of each point in $\phi(i+1)$. If a compactly generated $t$-structure on $\D(R)$ restricts to a $t$-structure on $\Dc(R)$ then the corresponding filtration satisfies the weak Cousin condition. If $R$ has a pointwise dualizing complex the converse is true. If the ring $R$ has dualizing complex then these are exactly all the $t$-structures on $\Dbc(R)$.\end{abstract}

\maketitle

\tableofcontents

\section*{Introduction}

The concept of $t$-structure on a triangulated category arises as a categorical framework for Goresky-MacPherson's intersection homology. Through this construction Be\u{\i}linson, Bernstein, Deligne and Gabber extended intersection cohomology to the \emph{\'etale} context. A $t$-structure provides a homological functor with values in a certain abelian category contained in the original triangulated category denominated the heart of the $t$-structure. In Grothendiecks's terms, this study accounts for the study of extraordinary cohomology theories for \emph{discrete} coefficients. Intersection cohomology and its variants have been studied successfully with these methods over the last twenty years. Let us point out \cite{BBD} and \cite{GNM} and references therein.

On the side of \emph{continuous} coefficients the development has proceeded at a slower pace. Deligne was first to make contributions to the problem of constructing  $t$-structures on the bounded derived category of coherent sheaves on a Noetherian scheme under certain hypothesis, namely the existence of a dualizing complex and of global locally free resolutions. His work was not available until the expository e-print by Bezrukavnikov \cite{bez}. Later Kashiwara \cite{Kashiwara} constructs from a decreasing family of supports satisfying certain condition a $t$-structure on the bounded derived category of coherent sheaves on a complex manifold, which corresponds in the algebraic case to a \emph{smooth} separated scheme of finite type over $\CPX$. Also, Yekutieli and Zhang \cite{yz} considered the Grothendieck  dual $t$-structure of the canonical one on the derived category of finitely generated modules over a Noetherian ring with dualizing complex. This $t$-structure is called the \emph{Cohen-Macaulay $t$-structure} in the present paper and it is shown to exist in the whole unbounded derived category.

Deligne, Bezrukavnikov and Kashiwara built, on the derived category of bounded complexes with finitely generated homologies, a $t$-structure starting with a finite filtration by supports $X=Z_s\supsetneq Z_{s+1} \supsetneq \cdots \supsetneq Z_{i}\supsetneq \cdots \supsetneq  Z_{n-1} \supsetneq Z_n=\varnothing$ in the corresponding topological space $X$. In Bezrukavnikov's paper Noetherian induction is used. Kashiwara constructs the triangle of the corresponding $t$-structure using in an essential way that the complexes are bounded. The filtrations by supports used by both authors are finite and satisfy a condition that we call in this paper \emph{the weak Cousin condition} (for any integer $i$, the set $Z_i$ contains all the immediate generalizations of each point in $Z_{i+1}$). This name refers to a weakening and reformulation of the notion of codimension function introduced by Grothendieck, see \cite[V.7]{RD} and \S \ref{wccondition} below. An equivalent notion was used in \cite{bez} under the name \emph{comonotone} perversity.

In the aforementioned papers, the authors rely on finite step-by-step constructions and do not take into account the possibility allowed by infinite constructions if one considers the unbounded derived category. This approach makes sense after \cite{AJS2} where it is proved that for a collection of objects in the unbounded derived category of a Grothendieck category there is a $t$-structure whose aisle is generated by this set of objects. 

A logical next step is to try to classify $t$-structures on $\D(R)$ by filtrations of subsets of $\spec(R)$ extending the fact, proved in the key paper \cite{Nct}, that Bousfield localizations (the class of triangulated $t$-structures) are classified by subsets of $\spec(R)$. However, a counterexample by Neeman and further developments by the third author made clear that residue fields of prime ideals were not the right objects to use in order to achieve the classification of $t$-structures (see remark on page~\pageref{ex-Ne-Ma}). 

  Stanley in his preprint \cite{Stanley} treated the problem of studying $t$-structures on $\Dbc(R)$, the subcategory of $\D(R)$ of bounded complexes with finitely generated homologies, where $R$ is a commutative Noetherian ring. He showed that it is not possible to classify all $t$-structures on $\D(R)$ because the class of $t$-structures on $\D(\ZZ)$ is \emph{not} a set \cite[Corollary~8.4]{Stanley}. Then it is not possible to put $t$-structures on $\D(R)$ in correspondence with collections of subsets of $\spec(R)$.
   
On the positive side, Stanley showed that there is an order preserving bijection between filtrations of $\spec(R)$ by stable under specialization subsets and nullity classes in $\Dbc(R)$.
Theorem B in \cite{Stanley} states that the weak Cousin condition is a necessary condition over a filtration for the corresponding nullity class in $\Dbc(R)$ to be an aisle, and he conjectured that the converse is true. 
However, the proof of Theorem~B in \cite{Stanley} does not seem to be complete. 
We give here an alternate approach to Stanley's result encompassing the unbounded category and, in addition, we answer Stanley's conjecture in the affirmative. 

Specifically, the category $\Dc(R)$ is skeletally small so any $t$-structure on $\Dc(R)$ is the restriction of a $t$-structure on $\D(R)$ (see Lemma~\ref{skeletallysmall} and Proposition~\ref{nuevo}).
We look at $t$-structures on $\D(R)$ generated by complexes with finitely generated homologies and characterize those that restrict to $t$-structures on $\Dc(R)$ ---and in general to any of the subcategories $\Dc^\sharp(R)$ for any boundedness condition 
 $\sharp\in\{+,-,\bb, \text{``blank''}\}$. We obtain the classification of all $t$-structures on $\D(R)$ generated by complexes in $\Dcmenos(R)$ in terms of filtrations by supports of $\spec(R)$.
They are exactly all compactly generated $t$-structures (Theorem~\ref{classification-c-g}). We prove that the filtration associated to a compactly generated $t$-structure on $\D(R)$ that restricts to a $t$-structure on $\Dc(R)$ (or in any of the above subcategories $\Dc^\sharp (R)$) necessarily satisfies the weak Cousin condition (Corollary~\ref{Cousin}). So we give a proof of \cite[Theorem~B]{Stanley} by different means.

The compact objects in $\D(R)$ are the perfect complexes. Translated to our context, Stanley's question asks whether $t$-structures generated by perfect complexes on $\Dc(R)$ (or, in general, in $\Dc^\sharp(R)$) are those determined by filtrations satisfying the weak Cousin condition. 
In the last section of our paper we answer this question in the affirmative for $\Dc(R)$ when the ring $R$ possesses a pointwise dualizing complex. For $\Dbc(R)$ we get an affirmative answer when the ring $R$ has a dualizing complex, a mild hypothesis already present ---as we have recalled--- in previous works.
\\

Let us describe the contents of this paper. In the first section we start recalling the notations and definitions used in this paper. We introduce the notion of \emph{total pre-aisle} and show that any total pre-aisle of $\Dc^\sharp(R)$ is the restriction of an aisle of $\D(R)$, with $\sharp$ as before. 
We also study how total pre-aisles behave under change of rings.

A filtration by supports of $\spec(R)$ is a decreasing family $ \dots \supset \phi(i) \supset \phi(i+1) \supset \dots$ of stable under specialization subsets of $\spec(R)$. Each filtration by supports $\phi\colon \ZZ \to\CP(\spec(R))$ has an associated aisle $\CU_\phi$, generated by $\{R/\ip[-i] \,;\,i\in\ZZ \text{ and } \ip\in\phi(i)\}$. In \S \ref{aislesdeterminedbyfil} we consider aisles generated by suspensions of cyclic modules. They are precisely the aisles associated to filtrations by supports. We characterize in section \ref{compactas} the aisles of $\D(R)$ generated by complexes in $\Dcmenos (R)$ in terms of homological supports (Proposition~\ref{4.7}). They correspond to compactly generated aisles of $\D(R)$ (Theorem~\ref{4.9}). 

We study in \S\ref{wccondition} the filtrations by supports of $\spec(R)$ that provide aisles in $\Dc^\sharp (R)$.  We note that all $t$-structures on $\Dcmenos(R)$ and $\Dbc(R)$ are generated by perfect complexes. Theorem~\ref{5.5} is the first main result in this section. It says that the \emph{weak Cousin condition} is necessary on a filtration by supports on $\spec(R)$ in order to restrict the corresponding compactly generated $t$-structure on $\D(R)$ to a $t$-structure on  $\Dc(R)$.
This theorem corresponds to Stanley's Theorem 7.7\footnote{The weak Cousin condition here corresponds to being comonotone in \cite{Stanley}.}. We deal with the unbounded category $\Dc(R)$ and obtain from this the result\begin{footnote}{We should remark that the statement of Theorem~\ref{5.5} is a variation of \cite[Proposition~7.4]{Stanley}, the key ingredient in \cite[Theorem~B]{Stanley}.}\end{footnote} for $\Dbc(R)$. Next we describe the filtrations by supports of $\spec(R)$ satisfying the weak Cousin condition (Proposition~\ref{discreteness-of-filtration} and Corollary~\ref{discreteness-of-filtration2}).  As a consequence, if $\spec (R)$ is connected and $R$ has finite Krull dimension then filtrations that satisfy the weak Cousin condition are finite and exhaustive. Therefore, there is a bijection between $t$-structures on $\Dcmenos(R)$ and $\Dbc(R)$ (Corollary~\ref{bounded=right-bounded}). 

In \S\ref{descripciontruncados} we give a description of the truncation functors associated to a \emph{finite} filtration. Proposition~\ref{dos-pasos} shows that the aisle associated to a two-step filtration by supports that satisfies the  weak Cousin condition restricts to a $t$-structure on $\Dc(R)$, for any Noetherian ring $R$. With all these tools at hand, in the last section we prove the remaining main result (Theorem~\ref{Stanley-conjecture}). Our strategy of proof is related to the one used in \cite{Kashiwara}.
Namely, this Theorem asserts that if $R$ possesses a dualizing complex, then the aisles of $\Dbc(R)$ are exactly those induced by filtrations by supports satisfying the  \emph{weak Cousin condition}. As a consequence we obtain a bijection between aisles of $\Dc(R)$ and filtrations satisfying the weak Cousin condition under the weaker hypothesis that $R$ possesses \emph{pointwise} dualizing complex. The existence of a dualizing complex on $R$ is a very mild condition. It is satisfied by all rings of finite Krull dimension that are quotients of a Gorenstein ring. This is the case for all finitely generated algebras over a regular ring (\emph{e.g.} over a field or over $\ZZ$). 

\section{Notation and preliminaries on $t$-structures}

\noindent\textbf{Notation and Conventions.}
All rings in this paper will be commutative and Noetherian. Given a prime ideal $\ip\in\spec(R)$, $k(\ip)$ stands for the residue field of $\ip$ and $R_\ip$ for the localization of $R$ with respect to $\ip$. The support of an $R$-module $N$ is the set of prime ideals $\supp(N)=\{\ip\in\spec(R)\,/\, N_\ip=N\otimes R_\ip \neq 0\}$. For an ideal $\ia\subset R$ we denote by $\V(\ia) := \{\ip\in\spec(R)\,\,;\,\, \ia\subset\ip\}$.

As usual $\Mod(R)$  denotes the category of modules over a ring $R$,  $\CCC(R)$ the category of complexes of $R$-modules, $\K(R)$ its homotopy category and $\D(R)$  its derived category. For complexes we use the upward gradings. Let $n$ and $m$ be integers, as usual
$\D^{\sharp}(R)\subset \D(R)$ denotes the full subcategory of those complexes whose homologies satisfy one of the standard boundedness conditions $\sharp\in\{\,\leq n, \, < n, \,\geq n,\, > n,\, +, \,-\,\}$, $\Dcacb{n}{m}(R):=\Dgeq{n}(R)\cap \Dleq{m}(R)$ and $\Db(R)=\Dmenos{}(R)\cap \Dmas{}(R)$. 
Let $\Dc(R)\subset \D(R)$ be the full subcategory of complexes with finitely generated homologies. The symbol $\Dc^\sharp(R)$  stands for $\Dc(R)\cap \D^\sharp(R)$ for any superscript $\sharp$.
\\

\noindent\textbf{Basics on $t$-structures.}
Let $\T$ be any triangulated category. We will denote by $(-)[1]$ the translation auto-equivalence of $\T$ and its iterates by $(-)[n]$, with $n \in \ZZ$.

A $t$-structure on $\T$ in the sense of Be\u{\i}linson, Bernstein, Deligne and Gabber (\cite[D\'efinition~1.3.1]{BBD}) is a couple of full
subcategories $(\CU,\CF[1])$ such that  $\CU[1] \subset \CU$, $\CF[1] \supset \CF$, $\Hom_{\T}(Z,Y) = 0$ for $Z \in \CU$ and $Y \in \CF$ (\ie $\,\,\CU\subset {}^{\perp}\CF$, equivalently $\CF\subset \CU^{\perp}$), and for each $X \in \T$ there is a distinguished triangle
\begin{equation}
\label{t-triangle}
	\tleq{}{\CU}X \lto X \lto \tg{}{\CU}X \overset{+}{\lto}
\end{equation}
with $\tleq{}{\CU}X \in \CU$ and $\tg{}{\CU}X \in \CF$.
The subcategory $\CU$ is called the aisle of the $t$-structure, and $\CF$ is called the co-aisle. It follows from the definition that $\CU= {}^{\perp}\CF$ and $\CF= \CU^{\perp}$. It also follows that the inclusion $\CU \inc \T$ has a right adjoint $\tleq{}{\CU}$ called the left truncation functor and, dually, the inclusion $\CF \inc \T$ has a left adjoint functor $\tg{}{\CU}$,  the right truncation functor. In fact, $X\in \CU$ if and only if $\tg{}{\CU}X=0$, similarly $X\in \CF$ if and only if $\tleq{}{\CU}X=0$. The $t$-structure can be described just in terms of its aisle $\CU$. This fact justifies the notation for the truncation functors.

We call $(\Dleq{0}(R), \Dg{0}(R)[1])$ the \emph{canonical} $t$-structure on $\D(R)$. With $n\in\ZZ$, the $t$-structures $(\Dleq{n}(R), \Dg{n}(R)[1])$ obtained by translations of the canonical one are called \emph{standard} $t$-structures on $\D(R)$.
As usual $\tleq{n}{}=\tl{n+1}{}$ and  $\tgeq{n+1}{}=\tg{n}{}$ denote the left and right truncation functors associated to the $n$-th standard $t$-structure. For each $\Xdot\in\D(R)$,
  \[
   \tleq{0}{}\Xdot \lto \Xdot \lto \tg{0}{}\Xdot \overset{+}{\lto}
  \]
denotes the distinguished triangle determined by the canonical $t$-structure.

\begin{cosa}
\label{1.1}
A class $\CU\subset\T$ is a \emph{pre-aisle} of $\T$ if $\CU$ endowed with the class of  distinguished triangles in $\T$ with vertices in $\CU$ is a suspended category in the sense of Keller and Vossieck \cite{KVsus}, that is, $\CU$ is a class closed for extensions such that $\CU[1]\subset\CU$. A pre-aisle $\CU$ of $\T$ is \emph{total} if $\CU={{}^\perp(\CU^\perp)}$ (orthogonal always taken in $\T$).

If  $\CU\subset\T$ is a class of objects such that $\CU[1]\subset\CU$ then the class ${{}^\perp(\CU^\perp)}$ is a total pre-aisle of  $\T$, it is the smallest total pre-aisle of  $\T$ containing $\CU$. The property $\CU[1]\subset\CU$ implies $\CU^\perp\subset \CU^\perp[1]$. In general, given a class $\CY\subset\T$ such that $\CY\subset \CY[1]$ then the class ${}^\perp\CY\subset\T$ is a total pre-aisle of $\T$. As a consequence, if $\T'$ is a triangulated subcategory of $\T$ and $\CV$ is a total pre-aisle of $\T'$ then  $\CV = \CU\cap\T'$ where $\CU$ is a total pre-aisle of $\T$.

An aisle (or in general, a total pre-aisle) $\CU\subset \T$ is \emph{generated} by a set of objects $\CW \subset \T$ if $\CU$ is the smallest aisle (total pre-aisle) of $\T$ containing ${\CW}.$ We will say that a $t$-structure on $\T$ is \emph{generated} by the set of objects $\CW$ if so is its aisle.


\end{cosa}

\begin{cosa}
\label{aisles}
Among the triangulated categories we are concerned with, only the unbounded category $\D(R)$ has coproducts. Starting with a family of objects in $\D(R)$ it is possible to construct its associated $t$-structure on $\D(R)$ as follows.
Given a set of objects $\CM\subset \D(R),$ let ${\CM}[\NN]:=\{\Mdot[i]\,;\,\,\Mdot\in\CM\,\text{ and }\, i\geq 0\}\subset \D(R).$ By \cite[Proposition 3.2]{AJS2}  the smallest cocomplete pre-aisle containing the objects in $\CM$ is an aisle, that we will denoted here by $\aisle{\CM}$. Note that $\aisle{\CM}=^\perp({\CM}[\NN]^\perp)$. Furtheremore we can always assume that $\aisle{\CM}$ is generated by a single object because $\aisle{\CM}=\aisle{M}$ with $M:=\oplus_{M\in \CM} M$.

Let us fix a superscript $\sharp\in\{\text{``blank''}, \,+,\,-,\,\bb\}$. We know that any total pre-aisle of $\Dc^\sharp (R)$ is the restriction of a total pre-aisle of $\D(R)$; Proposition~\ref{nuevo} below provides a more useful result.

\begin{lem}
\label{skeletallysmall}
The categories $\Dc^\sharp (R)$ are skeletally small.
\end{lem}

\begin{proof}
Let us treat first the case $\sharp = -$. By using step by step free resolutions and taking into account that $R$ is noetherian, we see that every object in $\Dc^- (R)$ is isomorphic to a bounded above complex of finitely generated free modules and they form a set $\CW$ that contains a representative for every complex in $\Dc^- (R)$.

The rest of the cases will be settled if we show that $\Dc(R)$ is skeletally small \ie\/, the case $\sharp = \text{``blank''}$. Let $\Xdot \in \Dc(R)$. Note that 
\[\Xdot \iso \dirlim{n \in \NN}\tleq{n}{}\Xdot.\]
Now every $\tleq{n}{}\Xdot$ is quasi-isomorphic to an object $\Wdot_n$ in $\CW$. On the other hand \cite[Proof of Lemma 3.5]{AJS1} there is a quasi-isomorphism
\[
\dirlim{n \in \NN}\tleq{n}{}\Xdot \iso \holim{n \in \NN}\tleq{n}{}\Xdot.
\]
Summing up $\Xdot$ is isomorphic to the cone of an endomorphism of $\bigoplus_{n \in \NN}\Wdot_n$. But it is clear that the collection of endomorphisms of countable coproducts of objects in $\CW$ form a set $\CM$ and this set contains a representative for every complex in $\Dc(R)$.
\end{proof}

\begin{prop}
\label{nuevo}
Let $\T^\sharp$ be any of the categories $\Dc^\sharp (R)$. Let $\CV$ be a total pre-aisle of $\T^\sharp$, and let $\CE$ be its right orthogonal in $\T^\sharp$. Then
\begin{enumerate}
\item $\CU={}^\perp(\CV^\perp)$ is an aisle of $\D(R)$ (we are using the symbol ${}^\perp$ for the orthogonal in $\D(R)$);
\item  the corresponding $t$-structure $(\CU,\CF[1])$ on $\D(R)$ satisfies that $\CV=\CU\cap \T^\sharp$ and $\CE=\CF\cap \T^\sharp$;
\item if $(\CV,\CE[1])$ is a $t$-structure on $\T^\sharp$ then for any $\Xdot\in \T^\sharp$ the distinguished triangle in $\T^\sharp$ defined by the $t$-structure $(\CV,\CE[1])$
  \[
   N\lto X
   \lto B \stackrel{+}{\lto}
  \]
is the distinguished triangle in $\D(R)$ associated to $(\CU,\CF[1])$.
\end{enumerate}
\end{prop}
\begin{proof} By the previous lemma the category $\T^\sharp$ is skeletally small, then we can choose a set of objects $\CW\subset\CV$ such that for each object in $\CV$ there is an isomorpic object in $\CW$. Then $\CE$ is the right orthogonal of $\CW$ in $\T^\sharp$. The class $\CU={}^\perp(\CW^\perp)$ is the aisle of $\D(R)$ generated by $\CW$, and trivially $\CV\subset \CU$. Let $\CF:=\CU^\perp$, that is $\CF=\CW^\perp$. In particular $\CE=\CF\cap \T^\sharp\subset\CF$ and therefore $\CU\cap \T^\sharp={}^\perp\CF\cap \T^\sharp\subset {}^\perp\CE\cap \T^\sharp= \CV$.  Whence $\CU\cap \T^\sharp=\CV$.
The last assertion in the proposition is obvious because $\CV\subset\CU$ and $\CE\subset\CF$.
\end{proof}

\end{cosa}

\begin{cosa}
\label{compact}
Let $\T$ be a triangulated category with coproducts. An object $E$ of $\T$ is called \emph{compact} if the functor $\Hom_{\T}(E,-)$ commutes with arbitrary coproducts. By Rickard's criterion the compact objects of $\D(R)$ are the perfect complexes, {\ie} those complexes isomorphic to bounded complexes of finite-type projective modules  (see \cite[Proposition 6.3 and its proof]{Ric})\begin{footnote}{In \cite[Lemma~4.3]{AJS2} the reader can find a simpler proof of this fact using results of Neeman \cite[Lemma~2.2]{Ntty}.}\end{footnote}.

We will say that an aisle (or in general, a total pre-aisle) $\CU\subset \T$ is \emph{compactly generated} if there is a set $\CE \subset \CU$ of compact objects in $\T$ such that $\CE$ generates  $\CU$. We will say that a $t$-structure is \emph{compactly generated} if its aisle is compactly generated. 

\begin{ex}
$\Dleq{0}(R)$ is a compactly generated aisle of $\D(R)$, it is generated by the stalk complex $R=R[0]$.
\end{ex}

\end{cosa}

\noindent\textbf{Compactly generated Bousfield localizations on $\D(R)$}.
\label{1.4}
In general the aisle of a $t$-structure on a triangulated category $\T$ is not a triangulated subcategory. In fact, an aisle $\CU$ of $\T$, 
$\CU$ is a triangulated subcategory of $\T$ if and only if  $\CU[-1]\subset\CU$, equivalently  the left truncation functor $\tleq{}{\CU}$ (equivalently, right truncation functor $\tg{}{\CU}$) is a $\Delta$-functor.  

A $t$-structure whose aisle $\CU$ is a triangulated subcategory of the ambient triangulated category $\T$ is called a \emph{Bousfield localization} of $\T$, and the class $\CU$ is called a localizing subcategory of $\T$. The objects in  $\CU$ are called acyclic and the functor $\tleq{}{\CU}$ is the associated acyclization functor. The objects in $\CU^{\perp}$ are called local objects and the functor $\tg{}{\CU}$ is called the Bousfield localization functor. For a reference on Bousfield localizations in this context, see \cite{AJS1}. For its classification see \cite{Nct} for rings and \cite{AJS3} for schemes and formal schemes.


\begin{cosa}
\label{util0}
Proposition~5.7 in \cite{AJS3} shows that compactly generated  Bousfield localizations on $\D(R)$ correspond to stable under specialization subsets of $\spec(R)$. Let us recall in our context some facts about compactly generated  Bousfield localizations from \cite{AJS3} that were obtained following the path initiated in \cite{AJL}.
 
%
%
%
A subset $Z \subset \spec(R)$ is stable under specialization if for any couple of prime ideals $\ip\subset \iq$ with $\ip\in Z$, it holds that $\iq\in Z$, in other words, it is the union of a directed system of closed subsets of $\spec(R)$. From now on, to abbreviate, we will refer to this kind of subsets as \emph{sp-subsets}. As usual, for each $R$-module $N$ let us denote by $\Gamma_Z(N)$  the biggest submodule of $N$ whose support is contained in $Z$.  The functor $\Gamma_Z \colon \Mod(R)\to \Mod(R)$ is an idempotent kernel functor, thus it is determined by its Gabriel topology\begin{footnote}{See  \cite[Ch. VI, \S 5]{St}.}\end{footnote} of ideals: the set of ideals $\ia\subset R$ such that $\supp(R/\ia)=\V(\ia)\subset Z$. Namely,
  \[
   \Gamma_Z :=
   \dirlim{\V(\ia)\subset Z} \Hom_{R}(R/\ia, -).
  \]
A basis of ideals of the Gabriel topology suffices to compute $\Gamma_Z$. Let $Q_Z \colon \Mod(R) \to \Mod(R)$ be the (abelian) localization functor associated to $\Gamma_Z.$ The canonical transformations $\Gamma_Z\to \id$  and  $\id \to Q_Z$  induce isomorphisms $ \Gamma_Z\Gamma_Z\cong \Gamma_Z$ and $Q_Z\cong Q_Z Q_Z$.
 
   Using K-injective resolutions these relations can be extended to the derived category $\D(R)$. In such a way that 
  \begin{equation}
  \label{abovetriangle}
   \R\Gamma_Z E \xto{\rho(E)}  
   E \lto \R{}Q_Z E \overset{+}\lto     
  \end{equation}
is the Bousfield localization triangle whose localization functor is $\R Q_Z$  and its acyclization functor $\R\Gamma_Z$ (see \cite[\S 2 and the example in page 16]{AJS3} for the results mentioned in this paragraph).
Moreover, for all $E \in \D(R)$ the natural map $E \otimes^{\LL}_{R} \R\Gamma_Z R \to  E$ factors through $\R\Gamma_Z E$, providing a canonical isomorphism
  \(
    E \otimes^{\LL}_{R} \R\Gamma_Z R \iso
   \R\Gamma_Z E
  \)
\cite[\S 2.3]{AJS3} in such a way that the above triangle (\ref{abovetriangle}) is canonically isomorphic to the triangle
  \[
   E \otimes^{\LL}_{R} \R\Gamma_Z R
 \xto{E \otimes^{\LL}_{R} \rho(R)}  
   E \lto E \otimes^{\LL}_{R}\R{}Q_Z R \overset{+}\lto     
  \]   
\cite[\S 2.1]{AJS3}. These properties are summarized in 
\cite[\S 5, p. 603]{AJS3} by saying that the corresponding Bousfield localization is \emph{$\otimes$-compatible}. In \lc{} it is proved that these are exactly the Bousfield localizations on $\D(R)$ whose localization functors commute with coproducts, {\ie} the \emph{smashing} localizing subcategories of $\D(R)$ ({\cfr} \cite[\S 3]{Nct}).
\end{cosa}

\begin{cosa}
\label{util}
As a direct consequence of the above results given  $Z_1, \,Z_2 \subset \spec(R)$ two sp-subsets, then:
  \begin{enumerate}
     \item The canonical transformation $\R\Gamma_{Z_1\cap Z_2}
           \to \R\Gamma_{Z_2}$ induces a natural isomorphism
           $\R\Gamma_{Z_1\cap Z_2} \to \R\Gamma_{Z_1}
           \R\Gamma_{Z_2}$.
     \item The canonical map of functors $\R Q_{Z_1} \R\Gamma_{Z_2}
           \to \R Q_{Z_1}$ induces a natural isomorphism
           $\R Q_{Z_1}\R\Gamma_{Z_2}\iso\R\Gamma_{Z_2}\R Q_{Z_1}$. Furthermore $\R Q_{Z_1}\R Q_{Z_2}$ and $\R Q_{Z_2}\R Q_{Z_1}$ are canonically isomorphic, and they are isomorphic to $\R Q_{Z_2\cup Z_1}$.
  \end{enumerate}
\end{cosa}

\begin{thm}
\label{suph}
Let $Z\subset \spec(R)$ be a sp-subset and $\Fdot\in\D(R)$. The canonical map $\R\Gamma_{Z} F  \to F$ is an isomorphism if and only if $\supp(\h^j(F)) \subset Z$, for every $j\in\ZZ$.
\end{thm}
\begin{proof} It is \cite[Theorem~5.6]{AJS3} translated into the present context.
\end{proof}

\begin{cor}
\label{ejemplo}
Let $Z \subset \spec(R)$ be a sp-subset and $i\in\ZZ$. The pre-aisle 
\[
\cu{i}{Z} :=\{\Ndot \in \Dleq{i}{}(R) \, ; \, \supp(\h^j(\Ndot))\subset Z \text{ for all } j\leq i \}
\]
 is an aisle of $\D(R)$ with $\tleq{i}{} \R\Gamma_{Z}$ as its associated left truncation functor.
\end{cor}

\begin{proof} By Theorem~\ref{suph}, $\cu{i}{Z}$ is the class of objects $\Ndot\in\D(R)$ such that $\tleq{i}{} \R\Gamma_{Z}\Ndot\cong \Ndot$. Then $\cu{i}{Z}$ is an aisle of $\D(R)$ and the right adjoint functor of the inclusion $\cu{i}{Z}\inc \D(R)$ is $\tleq{i}{} \R\Gamma_{Z}$.
\end{proof}

\noindent\textbf{Total pre-aisles and base change.}
Let $f\colon R\to A$ be a homomorphism of rings. The exact forgetful functor $f_*\colon\Mod(A) \to \Mod(R)$ has adjoints on both sides. The base change functor $f^*=A\otimes_R - $ is its left adjoint, and its right adjoint is $\Hom_R(A ,-).$ The derived functors 
  \[
   \LL f^*\colon \D (R) \to \D (A), \qquad
   \x{f} := \rhomdot_R(A ,-)\colon \D (R) \to \D(A),
  \]
defined using K-projective and K-injective resolutions in $\K(R)$ (see \cite[Theorem 2.14]{BN} and \cite{Sp}), satisfy the corresponding natural adjunction formulas:
  \[
   \Hom_{\D (R)}(M, f_*N)\cong
   \Hom_{\D (A)}(\LL f^* M, N),
  \]
  \[
   \Hom_{\D (R)}(f_*N, M)\cong
   \Hom_{\D (A)}(N, \R\Hom_R^\cdot(A , M) ),
  \]
for all $M \in \D (R),$ and $N \in \D (A)$. As it is usual, if there is no ambiguity, we will write $N=f_*N$ for every $N \in \D (A).$
The functor $f_*$  transforms acyclic complexes into  acyclic complexes, hence $\Hom_R^\cdot(A , -)\colon \K (R) \to \K (A)$ transforms a $K$-injective complex of (injective) $R$-modules into a $K$-injective complex of (injective) $A$-modules. As a consequence if  $g\colon A\to B$ is another morphism of rings then $\x{(g f)}=\x{f}\x{g}.$

Let $f\colon R \to A$ be a homomorphism of rings and let $\CW\subset \D(R)$ be a  class  of objects.  We denote by $\CW^{[f_*]}=\{\Ndot\in\D(A)\,;\, f_*\Ndot\in\CW\}$ the pre-image of $\CW$ through $f_*$, and by $\LL f^*\CW$ the image of $\CW$ through $\LL f^*$. We use the same orthogonal symbols for classes in $\D(R)$ and $\D(A)$ in each case the ambient category will tell us where the orthogonals are taken.

The following is a slightly more general reformulation of the statement in \cite[Corollary 5.2]{AJS2}  that is useful in the present context.

\begin{prop}
\label{rigid}
Let $\CU$ be a cocomplete pre-aisle of $\D(R),$ then $X \otimes^{\LL}_R M \in \CU$ for all $X \in \CU$ and $M \in \Dleq{0}{}(R).$ 
\end{prop}

\begin{proof}
The class $\CV= \{M \in \D(R) \,\,;\,\, X\otimes^{\LL}_R M \in \CU\}$ is a cocomplete pre-aisle of $\D(R)$ that contains $R=R[0],$ therefore $\Dleq{0}(R) \subset \CV.$
\end{proof}

\begin{cor}
\label{resumen}
Let $f\colon R \to A$ be a morphism of rings and let $\CU$ be a cocomplete pre-aisle of $\D(R)$. Then:
 \begin{enumerate}
   \item $f_*\LL f^* \CU \subset \CU$;
   \item $(\LL f^*\CU)^\perp{}=(\CU^\perp{})^{[f_*]}$;
   \item $\x{f}(\CU^\perp)\subset (\LL f^*\CU)^\perp{}$.
 \end{enumerate}
If furthermore $\CU$ is a total pre-aisle and $\CW:= {}^\perp{}((\LL f^*\CU)^\perp{})$, then $f_*\CW \subset \CU$.
\end{cor}

\begin{proof}
If $\Mdot\in\CU$ then Proposition~\ref{rigid} shows that $f_*\LL f^*\Mdot = A\otimes_R^{\LL} \Mdot \in\CU$, from which assertion $(1)$ follows.
Assertion $(2)$ follows immediately from the adjunction isomorphism
  \[
  \Hom_{\D(A)}(\LL f^*\Mdot ,\Ndot )\cong \Hom_{\D(R)}(\Mdot ,f_*\Ndot)
  \]
for all $\Mdot\in\D(R)$ and $\Ndot\in\D(A)$.
Due to $(1)$, for all $V\in \CU^\perp$ and  $U\in\CU$  we have that $\Hom_{\D(A)}(\LL f^* U ,\x{f}V  )
     \cong  \Hom_{\D(R)}(f_* \LL f^*U  ,V)=0$, that is $(3)$ follows.
     
In order to check the last assertion note that  for any $W\in\CW$ and  $V\in\CU^\perp$ one has that $0=\Hom_{\D(A)}(W ,\x{f} V)$ by $(3)$; therefore, 
\(
0 = \Hom_{\D(R)}(f_* W,V ).
\)
So $f_*\CW \subset {}^\perp(\CU^\perp)=\CU$.
\end{proof}

\begin{cosa}
\label{notacionlocalizacion}
Let $S\subset R$ be a multiplicative closed subset and $f\colon R\to S^{-1}R$ be the canonical ring homomorphism. The functor $f^*$ is exact so $\LL f^*=f^*\colon\D(R)\to\D(S^{-1}R)$. As usual we denote $f^*\Xdot= S^{-1}\Xdot$ for every object $\Xdot\in\D(R)$, and given a class $\CV$ in $\D(R)$,  $S^{-1}\CV$ stands for $f^*\CV$. The forgetful functor identifies $\D(S^{-1}R)$ with a full subcategory of $\D(R)$. Throughout the paper we identify $\spec(S^{-1}R)$ with the subset $\{\ip\in\spec(R)\,\,;\,\, \ip\cap S=\varnothing \}\subset \spec(R)$.  If $S = R \setminus \iq$ where $\iq$ is a prime ideal of $R$ we will write, as usual, $S^{-1}\CV = \CV_\iq$.

\end{cosa}

\begin{prop}
\label{ejemplogeo}
Let $Z \subset \spec(R)$ be a sp-subset, and let us fix $i$ an integer. Let us denote $\CU=\cu{i}{Z}$ (see Corollary~\ref{ejemplo}), and $\CF=\CU^\perp$. For any multiplicative closed subset $S\subset R$ the pair $(S^{-1}\CU, S^{-1}\CF[1])$ is a $t$-structure on $\D(S^{-1}R)$,  furthermore  $S^{-1}\CU=\CU\cap\D(S^{-1}R)$ and $S^{-1}\CF=\CF\cap\D(S^{-1}R)$.
\end{prop}

\begin{proof}
For every $\Mdot\in\D(R)$ the canonical map $S^{-1}\R\Gamma_Z \Mdot \to \R\Gamma_Z S^{-1}\Mdot$ is an isomorphism (see \S \ref{util}), therefore $\tleq{}{\CU} S^{-1}\Mdot \cong S^{-1} \tleq{}{\CU}\Mdot$. It follows that $S^{-1}\CU=\CU\cap\D(S^{-1}R)$  and $S^{-1}\CF=\CF\cap\D(S^{-1}R)$. 
Moreover, for any $M\in\D(S^{-1}R)$  the triangle in $\D(R)$
  \[
   \tleq{}{\CU}\Mdot \lto M \lto \tg{}{\CU}\Mdot \overset{+}{\lto}
  \] 
is also in  $\D(S^{-1}R)$, because $\tleq{}{\CU}\Mdot=\tleq{}{\CU}S^{-1}\Mdot\cong S^{-1}\tleq{}{\CU}\Mdot$. As a result $(S^{-1}\CU, S^{-1}\CF[1])$ is a $t$-structure on $\D(S^{-1}R)$.
\end{proof}

\begin{rem}
Note that in particular we have that $(S^{-1}\CU)^\perp =S^{-1}(\CU^\perp)$ (where the first orthogonal is taken in $\D(S^{-1}R)$ and the second in $\D(R)$).
\end{rem}

\section{Aisles
determined by filtrations of supports}
\label{aislesdeterminedbyfil}


\begin{cosa}
\label{ANTES-0.4}
Let us denote by $j_\ia \colon R\to R/\ia$  the canonical morphism determined by the ideal $\ia\subset R$. We begin this section describing the induced $t$-structures on $\D(R)$ by the standard $t$-structures on $\D(R/\ia)$ through the adjunction ${j_\ia}_*\dashv \x{j_\ia}$.

Let  $X$ be a complex of $R$-modules such that $\ia X^i=0$, for all $i\in\mathbb{Z}$. Then $X$ can be also viewed as a complex of $R/\ia$-modules, in such a way that ${j_\ia}_* X = X.$ That being so, for any $Y\in\D (R)$, there are isomorphisms
  \begin{equation}
  \label{ecuacion-0.4}
   \Hom_{\D(R)}(X ,Y )\cong
   \Hom_{\D(R/\ia)}(X , \x{j}_\ia Y)
  \end{equation}
\end{cosa}

\begin{lem}
\label{caract}
Let $\ia\subset R$ be an ideal, and $k\in\ZZ$. For a complex  $Y\in \D(R)$, the following statements are equivalent:
  \begin{enumerate}
   \item $Y \in \aisle{R/\ia[-k]}^\perp$;
   \item $\x{j}_\ia Y \in \Dg{k}(R/\ia).$
  \end{enumerate}
\end{lem}

\begin{proof}
As a consequence of \ref{ANTES-0.4} it holds that
  \begin{align*}
     \Hom_{\D(R)}({R}/{\ia}[i], Y)
     &\cong \Hom_{\D(R/\ia)}({R}/{\ia}[i],\x{j}_\ia Y)\\
     &=\h^{-i}(\rhomdot_{R/\ia}({R}/{\ia},\x{j}_\ia Y)) =\h^{-i}(\x{j}_\ia Y)
  \end{align*}
for all $i\in \ZZ$. So the result follows.
\end{proof}

\begin{lem}
\label{Two-ideals}
Let $\ia \subset R$ be an ideal. If $\Xdot\in \Dleq{0}{}(R/\ia)$, then $\Xdot={j_\ia}_* \Xdot \in \aisle{R/\ia}$.
\end{lem}

\begin{proof}
Clear.
\end{proof}

\begin{prop}
\label{prop-Two-ideals}
The following statements hold for any ideals $\ia ,\,\ib \subset R$:
  \begin{enumerate}
    \item If $\ia\subset\ib$ then $R/\ib\in\aisle{R/\ia}$.
    \item We have that $\aisle{R/{\ia \ib}}=\aisle{R/{\ia},R/{\ib}}
          =\aisle{R/{\ia\cap \ib}}$.
    \item For all $n \geq 1$, $\aisle{R/\ia^n}=\aisle{R/\ia}$.
    \item If $\rad(\ia)=\rad(\ib)$ then $\aisle{R/\ia}=\aisle{R/\ib}$.
  \end{enumerate}
\end{prop}

\begin{proof}
The statement $(1)$ is the particular case of Lemma~\ref{Two-ideals} in which $\Xdot\in \Dleq{0}{}(R/\ia)$ is the stalk complex $R/\ib$.
 
In order to prove $(2)$ note that $\ia/{\ia \ib}$ is also an $R/{\ib}$-module then $\ia/{\ia  \ib}\in\aisle{R/{\ib}}$ by Lemma~\ref{Two-ideals}. As a consequence the middle point in the exact sequence $0\to \ia/{\ia  \ib}\to R/{\ia  \ib}\to R/{\ia} \to 0$ belongs to $\aisle{R/{\ia}, R/{\ib}}$ because the extreme points do. Therefore $\aisle{R/{\ia  \ib}} \subset \aisle{R/{\ia}, R/{\ib}}$.  We finish the proof of $(2)$ applying $(1)$ to the chains of ideals $\ia  \ib \subset \ia \cap \ib \subset \ia$ and $\ia  \ib \subset \ia \cap \ib \subset \ib$. 

Statement $(3)$ follows by induction on $n\geq 1$ from the first equality in $(2)$.

The ring $R$ is Noetherian, so if $\rad(\ia)=\rad(\ib)$  there exist $s,t \in\NN$ such that $\ia^s\subset \ib$ and $\ib^t\subset \ia$. Therefore $(4)$ follows from $(1)$ and $(3)$.
\end{proof}

\begin{cor}
\label{corimportante2}
Let $\{\ip_1,\dots , \ip_s \}$ be the minimal prime ideals over the ideal $\ia \subset R$. Then $\aisle{R/\ia}=\aisle{R/\ip_1,\dots , R/\ip_s}$.
\end{cor}

\begin{proof}
Using that $\rad (\ia)=\ip_1\cap \dots\cap \ip_s$  and the second identity in Proposition~\ref{prop-Two-ideals}(2) we easily prove by induction on $s$ that $\aisle{R/\rad(\ia)}=\aisle{R/\ip_1,\dots ,R/\ip_s}$. Finally apply (4) in Proposition~\ref{prop-Two-ideals}. 
\end{proof}

\begin{cor}
\label{torsion}
Let $\ia\subset R$ be an ideal and $Z:=\V(\ia)\subset \spec(R)$. Let $Y \in\D(R)$ be a  complex  such that $Y\in \aisle{R/\ia}^\perp$, then $\R\Gamma_Z Y \in\Dg{0}(R).$
\end{cor}

\begin{proof}
Without lost of generality we may assume that $\Ydot$ is $K$-injective, so $\R \Gamma _Z Y=\Gamma _Z Y$. By Proposition~\ref{prop-Two-ideals} it holds that $Y\in \aisle{R/{\ia}^n}^\perp$ for all $n\geq 1$, that is, $\Hom_{\D(R)}({R}/{\ia^n}[i], \,\,Y)=0$ for all $n\geq 1$ and all $i\geq 0$. As a consequence for all $i\geq 0$      
  \begin{align*}
    \h^{-i} \Gamma _Z Y
    = & \h^{-i}\dirlim{n\geq 1}\Hom^\cdot_{R }({R}/{\ia^n},\,Y)\\
    = & \dirlim{n\geq 1}\h^{-i}\Hom^\cdot_{R }({R}/{\ia^n}, \,Y )\\
    = & \dirlim{n \geq 1} \Hom_{\D(R)}({R}/{\ia^n}[i], \,\,Y)=0. \qedhere
  \end{align*}
\end{proof}

\begin{prop}
\label{truncacionparaunideal}
Let $Z\subset \spec(R)$ be a sp-subset, $i\in\ZZ$, and $\cu{i}{Z}$ be the aisle defined in Corollary~\ref{ejemplo}. Then
  \[
   \cu{i}{Z}=\aisle{R/\ip[-i];\,\ip \in Z}.
  \]
\end{prop}

\begin{proof}
To prove this result it is enough to deal with the case $i=0$.
Let us recall from Corollary~\ref{ejemplo} that 
$\tleq{0}{} \R\Gamma_Z$ is the left truncation functor associated to the aisle $\cu{0}{Z}$. 

Let $\{\ia_\alpha\}_{\alpha \in I}$ be the Gabriel filter of ideals such that $Z_\alpha:=\V(\ia_\alpha)\subset Z.$ By Corollary~\ref{corimportante2} it is enough to prove that 
$
\cu{0}{Z}=\aisle{{R}/{\ia_\alpha}\,;\,\,\alpha \in I}
$.
Trivially $\CU := \aisle{{R}/{\ia_\alpha}\,;\,\,\alpha \in I} \subset \CU^0_Z.$  To prove the equality let us check that $\CU^\perp{}\subset {\CU^{0}_Z}{}^{\,\perp}$. Let $j\leq{0}$ and $Y \in \CU^\perp{},$ then Corollary~\ref{torsion}  asserts that $\h^j(\R\Gamma_{Z_\alpha} Y)=0$ for all $\alpha \in I$, therefore
  \[
   \h^j(\R\Gamma_Z Y)= 
   \dirlim{\alpha \in I} \h^j(\R\Gamma_{Z_\alpha} Y)=0
  \]
That is, $\tleq{0}{}\R\Gamma_Z Y=0$, equivalently $Y\in {\CU^0_Z}{}^{\,\perp}$.
\end{proof}

\begin{cosa}
\label{filtrationbyporaisle} A \emph{filtration by supports of} $\spec(R)$ is a decreasing map
  \[
   \phi \colon \ZZ\lto\CP(\spec(R))
  \]
such that $\phi (i)\subset \spec(R)$ is a sp-subset for each $i\in\ZZ$. To abbreviate, we will refer to a filtration by supports of $\spec(R)$ simply by a \emph{sp-filtration} of $\spec(R)$.

Let $\CU$ be an aisle of $\D(R)$. Having in mind that $\CU[1]\subset \CU$ and the statement $(1)$ in Proposition~\ref{prop-Two-ideals}, the aisle $\CU$ determines a sp-filtration
  $
   \phi_\CU\colon\ZZ\to\CP(\spec(R))
  $
by setting, for each $i\in\ZZ$,
  \[
  \phi_\CU(i):=\{\ip\in\spec(R)\,\,;\,\,R/\ip[-i]\in\CU\}.
  \]
The other way round a sp-filtration $\phi\colon \ZZ\to\CP(\spec(R))$ has an \emph{associated aisle}
  $
   \CU_\phi := \aisle{ {\,{R}/{\ip}[-i]\,\,;\,\, i\in\ZZ\,}
   \text{ and } \ip\in\phi(i)}.
  $
\end{cosa}

\bigskip

Fix $i$ an integer and $Z\subset\spec(R)$ a sp-subset. Let $\phi\colon \ZZ\to\CP(\spec(R))$ be the sp-filtration defined by $\phi(j)=Z$ for all $j\leq i$, and $\phi(j)=\varnothing$ if $j > i$. The previous proposition shows that $\cu{i}{Z}=\CU_\phi$.
The following shows the compatibility of these aisles with respect to localization in a multiplicative closed subset of $R$, generalizing
Proposition~\ref{ejemplogeo}:

\begin{prop}
\label{localizacion}
Let $S\subset R$ be a multiplicative closed subset. Given a sp-filtration $\phi\colon \ZZ\to\CP(\spec(R))$ let us denote by $\CF_\phi$  the right orthogonal of $\CU_\phi$ in $\D(R)$. Then $(S^{-1}\CU_\phi,\, S^{-1}\CF_\phi [1])$ is a $t$-structure on $\D(S^{-1}R)$, furthermore  $S^{-1}\CU_\phi=\CU_\phi\cap\D(S^{-1}R)$ and $S^{-1}\CF_\phi=\CF_\phi\cap\D(S^{-1}R)$.  Besides $S^{-1}\CU_\phi\subset \D(S^{-1}R)$ is the associated aisle to the sp-filtration
  \[
   \phi_S\colon \ZZ\to \CP(\spec(S^{-1}R))
  \]
defined by $\phi_S(i):=\phi(i)\cap \spec(S^{-1}R)$, for $i\in \ZZ$.
\end{prop}

\begin{proof}
The aisle $\CU:=\CU_\phi$ is the smallest containing all the aisles in the set $\{\CU_i:=\cu{i}{\phi(i)}\,\,;\,\,i\in\ZZ\}$, equivalently $\CF:=\CF_\phi$ is obtained by intersecting the classes $\{\CF_i:={\cu{i\,\,\,\perp}{\phi(i)}}\,\,;\,\,i\in\ZZ\}$. For any $M\in\D(S^{-1}R)$ the distinguished triangle in $\D(R)$ associated to $(\CU, \CF[1])$
  \begin{equation}
  \label{localt}
   N \lto M \lto Y \overset{+}{\lto}
  \end{equation}
belongs to $\D(S^{-1}R)$. Indeed, $S^{-1}\Ndot \in\CU$ by Proposition~\ref{rigid}. Given a complex $\Ydot\in \CF$ we have that $\Ydot\in \CF_i$ for any $i\in\ZZ$, then by Proposition~\ref{ejemplogeo}, we have that $S^{-1}\Ydot\in \CF_i$ for all $i\in\ZZ$, that is $S^{-1}\Ydot\in \CF$. Necessarily the distinguished triangle $S^{-1}N \lto M \lto S^{-1}Y \overset{+}{\lto}\,\,$ is canonically isomorphic to (\ref{localt}).
As a consequence $S^{-1}\CU$ is an aisle of $\D(S^{-1}R)$ with right orthogonal class $S^{-1}\CF$. Moreover  $S^{-1}\CU=\CU\cap\D(S^{-1}R)$  and $S^{-1}\CF=\CF\cap\D(S^{-1}R)$. Finally, note that $\Edot\in S^{-1}\CF$ if and only if, for each $i\in\ZZ$
  \[
  0=\Hom_{\D(R)}(R/\ip[-i], \Edot[j])\cong
     \Hom_{\D(S^{-1}R)}(S^{-1}(R/\ip)[-i], \Edot[j]).
  \]
for all $\ip\in \phi(i)$ and all $j\leq 0$. This fact amounts to saying that
\[
0=\Hom_{\D(S^{-1}R)}(S^{-1}R/\iq[-i], \Edot[j]),
\]
for each $i\in\ZZ$, all $\iq\in \phi(i)\cap \spec(S^{-1}R)$ and all $j\leq 0$. We conclude that $S^{-1}\CU$ is the aisle of $\D(S^{-1}R)$ associated to $\phi_S\colon \ZZ\to\CP(\spec(S^{-1}R))$.
\end{proof}

\begin{rem}
Let us consider the notation in Proposition~\ref{localizacion}. Let $\CW$ be any of the classes $\CU_\phi$ or ${\CF_\phi}$. From the previous results it follows that a complex $\Xdot\in\D(R)$ is in ${\CW}$ if and only if $\Xdot_\ip$ belongs to ${\CW}$ for any $\ip\in\spec(R)$.
\end{rem}

\section{Compactly generated aisles}
\label{compactas}

In Proposition~\ref{4.7} we study the aisles of $\D(R)$ generated  by bounded above complexes with finitely generated homologies.  It is a key result in the proof of Theorem~\ref{4.9} and Theorem~\ref{classification-c-g}, the main results in this section. We begin by proving some useful lemmas. Let us adopt the convention that $\Dleq{-\infty}(R) = 0$ and $\Dg{-\infty}(R) = \D(R)$.
 

\begin{lem}
\label{corrector}
Let $\ia\subset R$ be an ideal, $j_\ia \colon R\to R/\ia$ the canonical map, and $Y$ a complex of $R$-modules. Assume that the following two
conditions hold for a fixed $m\in \ZZ$:
  \begin{enumerate}
      \item $\x{j}_\ia Y \in \Dg{m}(R/\ia),$
      \item $\supp(\h^i(Y))\subset \V (\ia )$ for all $i\leq m.$
  \end{enumerate}
Then $\Ydot\in \Dg{m}(R).$
\end{lem}

\begin{proof}
It is enough to deal with the case $m=0$. Let $Z:=\V(\ia)$. By Theorem~\ref{suph} the hypothesis $(2)$ is equivalent to assuming that the canonical map 
  $
   \R \Gamma_Z  \,\tau^{\leq 0}Y \to \tau^{\leq 0}Y
  $
is an isomorphism. Furthermore, by Lemma~\ref{caract} and Corollary~\ref{torsion}, hypothesis $(1)$ implies that $\R\Gamma_Z Y \in\Dg{0}(R)$. Bearing in mind that $\R \Gamma_Z \Dgeq{0}(R)\subset \Dgeq{0}(R),$ it follows that the canonical map $\tau^{\leq 0}Y \to Y$ induces isomorphims
  $
   \h^{-i} (\R \Gamma_Z \,\tau^{\leq 0}Y) \to
   \h^{-i} (\R \Gamma_Z Y),
  $
for all $i\geq 0$
; and so 
  $
   \h^{-i}  (\tau^{\leq 0}Y) \osi \h^{-i} (\R \Gamma_Z
   \,\tau^{\leq 0}Y) \iso \h^{-i} (\R \Gamma_Z Y)=0
  $.
\end{proof}

\begin{lem}
\label{killed-by-the-maximal} 
Assume $R$ is local with maximal ideal $\im \subset R$. If $Y\in \D(R)$ is a complex such that
$\supp(\h^j(Y))\subset\{\im \}$, for all $j\in\ZZ$, then for any $X\in \Dc^-(R)$ the following are equivalent:
  \begin{enumerate}
     \item For all $i\geq 0$, $\Hom_{\D(R)}(X [i], Y)=0$.
     \item There is $n\in\ZZ\cup\{-\infty\}$, such that
           $X\in \Dleq{n}(R)$ and
           $Y\in \Dg{n}(R)$.
  \end{enumerate}
\end{lem}

\begin{proof}
The implication $(2)\Longrightarrow (1)$ is trivial. Assume $(1)$ for a non acyclic complex $X\in \Dc^-(R)$.  
Due to Proposition~\ref{rigid} we have that, for all $i\geq 0$,
\[
\Hom_{\D(R)}(X\otimes_R^{\LL} {R}/\im [i],Y)=0.\]
Let $n:=\max\{j\in\ZZ\,;\,\, \h^j(X )\neq 0\},$ then $\h^n(X\otimes_R^{\LL{}}{R}/{\im  })\cong \h^n(X )\otimes_R{R}/{\im}\neq 0$ by Nakayama's lemma. Since $X\otimes_R^{\LL{}}{R}/{\im}$ is a complex of ${R}/{\im}$-vector spaces we get
  \[
   \Hom_{\D(R/\im)}(X\otimes_R^{\LL{}}
   {R}/{\im} [i],\x{j}_{\im} Y )\cong
   \Hom_{\D(R)}(X\otimes_R^{\LL{}} {R}/{\im}[i],Y)=0,
  \]
for all $i\geq 0$. Notice that $\h^n(X\otimes_R^{\LL{}}{R}/{\im  })[-n]$ is isomorphic in $\D(R/\im)$ to a direct summand of $X\otimes_R^{\LL{}}{R}/{\im}$, therefore
  \[
   0=\Hom_{\D(R/\im)}(\h^n(X\otimes_R^{\LL{}}{R}/{\im  })[-n+i],\x{j}_{\im} Y ),
  \]
for any $i\geq 0$, which in turn implies that
  \[
   0=\Hom_{\D(R/\im)}({R}/{\im  }[i],\x{j}_{\im} Y),
  \]
for all $i\geq -n$. That is $\x{j}_{\im} Y \in \Dg{n}(R/\im)$. Now from Lemma~\ref{corrector} we can conclude that $Y\in \Dg{n}(R)$.
\end{proof}

\begin{cosa}
\label{notation-top-index}
Let us fix the convention that $\max(\varnothing) = \min (\varnothing)=-\infty$ for the empty subset  $\varnothing \subset \ZZ$. Then for $\Xdot\in \Dcmenos(R)$ and $\ip\in \spec(R)$ the following are well-defined elements in the set $\ZZ\cup\{-\infty\}$
  \begin{align*}
     m_\ip (\Xdot):=  & \max\{j\in\ZZ\,;\,\,
     \ip\in \supp(\h^j(\Xdot ))\}\\
     h_\ip (\Xdot):=  & \max\{j\in\ZZ
     \,;\,\,\supp(\h^j(\Xdot\otimes_R^{\LL}{R}/{\ip}))=\spec(R/\ip)\}.
  \end{align*}
\end{cosa}

\begin{lem}
\label{top-index}
For any $\Xdot\in \Dcmenos(R)$ and $\ip\in \spec(R)$, it holds that $m_\ip(\Xdot)= h_\ip(\Xdot)$.
\end{lem}

\begin{proof} 
Let us put $m=m_\ip (\Xdot)$ and $h=h_\ip (\Xdot)$. From the canonical isomorphisms
  \[
   (X\otimes_R^{\LL}R_\ip)\otimes_{R_\ip}^{\LL}k(\ip)\cong
   X\otimes_R^{\LL}k(\ip)\cong
   (X\otimes_R^{\LL}{R}/{\ip})\otimes_{R/\ip}^{\LL}k(\ip),
  \]
it follows that for any integer $j\in\ZZ$ such that $\Xdot_\ip\in\Dleq{j}(R_\ip)$ necessary $(\Xdot\otimes_R^{\LL}{R}/{\ip})\otimes_{R/\ip}^{\LL}k(\ip)\in\Dleq{j}(k(\ip))$. Then having in mind that $m$ and $h$ can be computed as $m=\min\{j\in\ZZ\,;\,\, \Xdot_\ip\in\Dleq{j}(R_\ip)\}$ and $h=\min\{j\in\ZZ\,;\,\, (\Xdot\otimes_R^{\LL}{R}/{\ip})\otimes_{R/\ip}^{\LL}k(\ip)\in\Dleq{j}(k(\ip)\}$ we get that $h\leq m$.

Trivially $h=m=-\infty$ when $\ip\not\in\bigcup_{i\in\ZZ}\supp(\h^i(X ))$. Assume that $\ip\in\bigcup_{i\in\ZZ}\supp(\h^i(X ))$. Then $X\otimes_R^{\LL}R_\ip\in \Dleq{m}(R_\ip)$ and $\h^m(X\otimes_R^{\LL} R_\ip)\cong \h^m(X )\otimes_R R_\ip\neq 0$. Hence $(X\otimes_R^{\LL}R_\ip)\otimes_{R_\ip}^{\LL}k(\ip)\in \Dleq{m}(k(\ip))$ and
  \[
   \h^m((X\otimes_R^{\LL}R_\ip)\otimes_{R_\ip}^{\LL}k(\ip))\cong
   \h^m(X\otimes_R^{\LL}R_\ip)\otimes_{R_\ip}k(\ip)
  \]
By Nakayama's lemma, the module $\h^m(X\otimes_R^{\LL}R_\ip)\otimes_{R_\ip}k(\ip)$ is nonzero, so $\h^m((X\otimes_R^{\LL}{R}/{\ip})\otimes_{R/\ip}^{\LL}k(\ip))$ is nonzero, hence $m= h$.
\end{proof}

\begin{lem}
\label{auxiliar2}
Let $R$ be a  commutative Noetherian  integral domain. Let $X\in \Dcmenos(R)$ be a complex and $0\in \spec(R)$ be the generic point. With the notation in \ref{notation-top-index}, $m_{0}(X)=\max\{ i\in\ZZ\,\,;\,\,\supp(\h^i(X))=\spec(R)\}$. If $Y\in \D(R)$ satisfies the following conditions:
  \begin{enumerate}
    \item $\Hom_{\D(R)}(X ,Y [i])=0$, for all $i\leq 0$, and
    \item for every $0\neq \ip\in\spec(R)$,
          $\rhomdot_R(R/\ip,Y)\in \Dg{m_0}(R)$,
  \end{enumerate}
then $Y\in \Dg{m_0}(R)$.
\end{lem}

\begin{proof} Let $K$ be the field of fractions of $R$ and let us set $Z:=\spec(R)\setminus \{0\}$ and $m:=m_{0}(X)$. Let us study the non trivial case, so assume that $m\in\ZZ$. The complex $X$ belongs to $\Dcleq{n}(R)$ for an integer $n\geq m$. By Proposition~\ref{truncacionparaunideal}, the hypothesis $(2)$ on $\Ydot$ is equivalent to saying  $\R\Gamma Y \in \Dg{m}(R)$, where $\Gamma:=\Gamma_Z \colon \Mod(R)\to \Mod(R)$ is the usual torsion radical. Applying the homological functor $\Hom(\Xdot,-):=\Hom_{\D(R)}(\Xdot,-)$ to the canonical triangle
  \begin{equation}
  	\label{llamar}
   \R\Gamma \Ydot \lto Y \lto 
   Y\otimes^{\LL}_R K\stackrel{+}{\lto}
  \end{equation}
we get exact sequences
  \[
   \Hom(\Xdot,Y[i-1]) \rightarrow \Hom(\Xdot,Y\otimes^{\LL}_R
   K[i-1]) \rightarrow \Hom(\Xdot,\R\Gamma \Ydot [i])
  \]
for all $i\in \ZZ$. Notice that $\Xdot\in\Dleq{n}(R)$, $\R\Gamma \Ydot [m-n]\in\Dg{n}(R)$ and $Y[-1]\in \aisle{\Xdot}^\perp$, then we get $0=\Hom_{\D(R)}(\Xdot,Y\otimes_R K[i-1])$ for $i\leq m-n \leq 0$. As a consequence
  \[
   \Hom_{\D(K)}(\Xdot \otimes_R K ,Y\otimes_R K[i-1])\cong
   \Hom_{\D(R)}(\Xdot,Y\otimes_R K[i-1])=0
  \]
for all $i\leq m-n$. Recall that $\Xdot \otimes_R K\in \Dcleq{m}(K)$ and $\h^{m}(\Xdot \otimes_R K)\neq 0$, therefore ${Y}\otimes_R K[m-n]\in\Dg{m} (K)$ since $K$ is a field and so Lemma~\ref{killed-by-the-maximal}  applies here. Then we conclude that  $Y\otimes_R K\in\Dmas{}(R)$. Therefore $Y\in\Dmas{}(R)$ by the existence of the distinguished triangle (\ref{llamar}). From this fact we are going to prove a more precise homological bound for $Y\otimes_R K$, namely $Y\otimes_R K\in \Dg{m}(K)$. Indeed, for all $i\in\ZZ$ there is a canonical isomorphism
  \[
   \Hom_{\D(K)}(X\otimes_R K,Y\otimes_R K[i])\cong
   \Hom_{\D(R)}(X ,Y [i])\otimes_R K,
  \]
since $Y\in\Dmas{}(R)$ and $\Xdot\in\Dcmenos(R)$. 
Then, by adjunction, hypothesis $(1)$ and Proposition~\ref{rigid}
\[
   \Hom_{\D(K)}(X\otimes_R K,Y\otimes_R K[i])\cong
   \Hom_{\D(R)}(X ,Y \otimes_R K[i]) =0,
  \]
for all $i\leq 0$. Hence $Y\otimes_R K\in \Dg{m}(K)$ by Lemma~\ref{killed-by-the-maximal}.  
Using once again the distinguished triangle (\ref{llamar}) we conclude $Y\in \Dg{m}(R)$ as desired.
\end{proof}

\begin{lem}
\label{all-are-geometric1}
Let $R$ be a commutative Noetherian ring. Let $X\in \Dcmenos (R)$ and $Y\in \D(R)$. If $\Hom_{\D(R)}(X ,Y [i])=0$ for all $i\leq 0$, then
  \[
   \Hom_{\D(R)}({R}/{\ip}[-k],Y [i])=0
  \]
for all $i\leq 0$, $k\in\ZZ$ and any $\ip\in \supp(\h^k(X ))$.
\end{lem}

\begin{proof}
As a consequence of Corollary~\ref{resumen}, it follows from the hypothesis that $\Hom_{\D(R)}(X\otimes_R^{\LL}{R}/{\ip} ,Y [i])=0$ for any $i\leq 0$, and  
  \[
   \Hom_{\D(R/\ip)}(X\otimes_R^{\LL}{R}/{\ip},
   \rhomdot_R({R}/{\ip},Y)[i])=0
  \]
for all $i\leq 0$ and every $\ip\in\spec(R)$.

Let us fix an integer $k\in\ZZ$ such that $\h^k(X )\neq 0$ and take any $\ip\in \supp(\h^k(X ))$. Then we have $k\leq m_\ip =h_\ip$, where $m_\ip=m_\ip(\Xdot)$ and $h_\ip=h_\ip(\Xdot)$ (see Lemma~\ref{top-index}).

We proceed by \emph{reductio ad absurdum} assuming that the set
\[
\CS:=\{\ip\in\spec (R) \,;\,\,\ip\in\supp(\h^k(X )) \text{ and } \Ydot\notin \aisle{R/\ip[-k]}^\perp \}
\]
is nonempty. The ring $R$ is Noetherian so we can choose a maximal element $\ip_0$ in $\CS$. Note that for $\Xdot_0:=X\otimes_R^{\LL}{R}/{\ip_0}\in\Dcmenos(R/\ip_0)$ the complex $\Ydot_0:=\rhomdot_R({R}/{\ip_0},Y)\in\D(R/\ip_0)$ satisfies the following properties:
  \begin{enumerate}
    \item $\Hom_{\D({R}/{\ip_0})}(\Xdot_0,\Ydot_0[i])=0$,
          for all $i\leq 0$ (see the remark at the beginning of the proof).
    \item If $\iq\in\spec(R)$ is such that $\ip_0\subsetneq \iq$, then
          $\iq\in \supp(\h^k(\Xdot))$ and
          $\iq\notin  \CS$ since $\ip_0$ is maximal
          in $\CS$. Therefore $\Ydot\in \aisle{R/\iq[-k]}^\perp$.
          This fact amounts to saying that
          $\rhomdot_{R/\ip_0}(R/\iq,\Ydot_0)\in \Dg{k}(R/\ip_0)$,
          since
          $\Hom_{\D(R/{\ip_0})}(R/\iq,Y_0[i])
          \cong\Hom_{\D(R)}(R/\iq,Y[i])$.
\end{enumerate}  
Then Lemma~\ref{auxiliar2} shows that $\Ydot_0\in\Dg{k}(R/\ip_0)$, 
so $\Hom_{\D(R)}({R}/{\ip_0}[i],Y) = 0$ for all $i\geq -k$. This fact contradicts the assumption  $\ip_0\in\CS$, then necessarily $\CS=\varnothing$, and the result follows.
\end{proof}

We are now ready to prove a key result in this section.

\begin{prop}
\label{4.7} 
Let $R$ be a commutative Noetherian ring. For $X \in \D^-_{\tf}(R)$ and $Y\in \D(R)$, the following are equivalent:
   \begin{enumerate}
      \item $\Hom_{\D(R)}(X, Y[i])=0,$
            for all $i\leq 0$;
      \item $\Hom_{\D(R)}(\h^j(X)[-j], Y[i])=0,$
            for any $j\in\ZZ$ and $i\leq 0$;
      \item $\Hom_{\D(R)}(R/\ip  [-j], Y[i])=0,$
            for all $j\in\ZZ$, $i\leq 0$ and all prime ideals $\ip$
            (minimal) in $\supp(\h^j(X ))$;
      \item $\Hom_{\D(R)}(R/\ip  [-j], Y[i])=0,$
            for every $j\in\ZZ$, $i\leq 0$ and all prime ideals $\ip$
            (minimal) in $\ass (\h^j(X ))$.
   \end{enumerate}
\end{prop}

\begin{proof} 
The equivalence  between $(3)$ and $(4)$ follows directly  from Proposition~\ref{prop-Two-ideals}(1). 

Lemma~\ref{all-are-geometric1} is just $(1)\imp (3)$.
To show that $(3)\imp (1)$, we only need to assume here that $X \in \Dmenos{}(R).$ For simplicity let us suppose that $\Xdot\in\Dleq{0}{}(R)$.
Let $\CU=\CU_\phi$ where $\phi\colon\ZZ\to\CP(\spec(R))$ is the sp-filtration  defined,  for each $i\in\ZZ$, by setting
$\phi(i):=\cup_{j\geq i} \supp(\h^j(\Xdot))$ ({\cfr} \ref{filtrationbyporaisle}). Item $(3)$ says that $\Ydot\in\CU^\perp$. So to prove $(1)$ is enough to check that $\Xdot \in\CU$. Let us consider the canonical triangle
  \[
   \tleq{ }{\phi}  X \lto X \lto  \tg{ }{\phi}X
   \overset{+}{\lto}
  \]
and denote $N =\tleq{ }{\phi}X$ and
$B=\tg{ }{\phi}X.$ We claim that $B=0,$
equivalently the canonical map $ N \to X$ is an
isomorphism whence we get the desired result
$X\cong N\in \CU$. Note that $B\in\Dleq{0}(R)$
because $ N$ and $X$ belong to $\Dleq{0}(R).$ If
$B\neq 0,$ let us choose $\iq\in\spec(R)$ minimal in the set
$\cup_{t\leq 0}\supp(\h^t(B))$. By localizing from the above triangle we get
the distinguished triangle in $\D(R_\iq)$
\begin{equation}
\label{trianguloanterior}
N_\iq \lto X_\iq \lto  B_\iq \overset{+}{\lto}
\end{equation}
Recall from Proposition~\ref{localizacion} that $N_\iq\in \CU_\iq,$ and $B_\iq\in
({\CU}^\perp)_\iq=(\CU_\iq)^\perp$ (notation as in \ref{notacionlocalizacion}). Let $b:=\max\{j\leq 0 \,/\,
\iq\in \supp (\h^j(B))\}$ and $Z:=\{\iq R_\iq\}\subset \spec(R_\iq)$.
Then $\R \Gamma_Z(\Bdot_\iq)\cong \Bdot_\iq\in\Dleq{b}(R_\iq)$, that is  $\tleq{b}{}\R \Gamma_Z(\Bdot_\iq)\cong \Bdot_\iq$. Hence $B_\iq\in \aisle{R_\iq/\iq
R_\ip[-b]}$ as a consequence of Proposition~\ref{truncacionparaunideal}. 
If $X_\iq=0$ then $B_\iq \cong N_\iq
[1]\in \CU_\iq$, so in this case $B_\iq=0$. Suppose that
$X_\iq\neq 0,$ and set $ m:=  \max\{j\leq 0 \,\,;\,\, \iq R_\iq\in \supp(\h^j(X_\iq))\} =  \max\{j\leq 0 \,\,;\,\, \iq \in \supp (\h^j(X))\}$. Notice that then $\Xdot_\iq\in\Dleq{m}(R_\iq)$. The aisle $\CU_\iq\subset\D(R_\iq)$ is generated by the set
$\{R_\iq/\ip R_\iq[-i]\,\,;\,\,\ip \in \supp(\h^i(\Xdot)), \, i\in\ZZ\}=\{R_\iq/\ip R_\iq[-i]\,\,;\,\,\ip \in \supp(\h^i(\Xdot_\iq)), \, i\in\ZZ\}$ (see Proposition~\ref{localizacion}). Hence $\CU_\iq$ is contained in $\Dleq{m}(R_\iq)$, and so $\Ndot_\iq\in\Dleq{m}(R_\iq)$. Using the triangle (\ref{trianguloanterior}) we get that $b\leq m$. Therefore $\aisle{R_\iq / \iq R_\iq[-b]}\subset\aisle{R_\iq / \iq R_\iq[-m]}\subset \CU_\iq$. Then $\Bdot_\iq\in \CU_\iq$, and again as a result $B_\iq=0.$

Finally the equivalence between $(2)$ and $(3)$ is a consequence of $(1) \dimp (3)$ for the complex
  \[
   \bigoplus_{j\in \ZZ} \h^j(X)[-j]\in \Dc^-(R).\qedhere
  \]
\end{proof}

\begin{cor}
\label{worth}
For any $\Xdot \in \D^-_{\tf}(R)$,
\begin{align*}
\aisle{\Xdot}= & \,\,\aisle{\h^j(X)[-j] ; j\in\ZZ}\\
             = & \,\,\aisle{R / \ip[-j] ; j\in\ZZ \text{ and }\ip\in\supp \h^j(X)}.
\end{align*}
\end{cor}
\begin{proof} Immediate from the above proposition.
\end{proof}

\begin{cor}
\label{corimportante}
Let $i$ be an integer and let $Z$ be a sp-subset of $\spec(R)$. The aisle $\CU^i_Z\subset D(R)$ is compactly generated.
\end{cor}

\begin{proof}
It is enough to discuss the case $i=0.$ For each ideal $\ia\subset R$ such that  $\V(\ia)\subset Z$, let us fix a system of generators $\{a_1,\ldots,a_r\}$ of $\ia$. Let $K_\cdot(a_1,\ldots,a_r)$ be the  Koszul complex associated to the sequence $\{a_1,\ldots,a_r\}$ (\cite[III, (1.1.1)]{EGA}). Recall that $K_\cdot(a_1,\ldots,a_r)$ is the complex of $R$-modules defined by
  \[
   K_\cdot(a_1,\ldots,a_r):=\otimes_{j=1}^r K_\cdot(a_j),
  \]
where $K_\cdot(a_j)$ is the complex $(0)$ in all degrees apart from degrees $-1$ and $0$, and whose differential in degree $-1$ is $R\stackrel{a_j}{\lto}R$ the map multiplying by $a_j$. The complex $K_\cdot(a_1,\ldots,a_r)$ is a complex of finitely generated free modules in degrees $[-r,0]$ and $0$ elsewhere, and whose homologies are killed by the ideal $\ia$. The complex $K_\cdot(a_1,\ldots,a_r)$ is compact ({\cfr} \ref{compact}). Furthermore $\supp(\h^i(K_\cdot(a_1,\ldots,a_r)))\subset \V(\ia)$ and $\h^0(K_\cdot(a_1,\ldots,a_r))= R/\ia$. Therefore, Proposition~\ref{4.7}  shows that the aisle generated by the family of complexes
  \[
   \{K_\cdot(a_1,\ldots,a_r)\,;\,\,
   \{a_1,\ldots,a_r\}\subset R \text{ and }
   \V(\langle{a_1,\ldots,a_r}\rangle)\subset Z\}
  \]
agrees with $\cu{0}{Z}$.
\end{proof}

We are now ready to state and prove the main results in this section.

\begin{thm}
\label{4.9}
Let $R$ be a commutative Noetherian ring and $(\CU,\CF[1])$ be a
t-structure on $\D(R)$. The following assertions are equivalent:
  \begin{enumerate}
   \item $\CU$ is compactly generated;
   \item $\CU$ is generated by stalk
         complexes of finitely generated (resp. cyclic) $R$-modules;
   \item $\CU$ is generated by complexes in $\Dbc(R)$;
   \item $\CU$ is generated by complexes in $\Dcmenos(R)$;
   \item there exists a sp-filtration
         $\phi \colon \ZZ\to\CP(\spec(R))$
         such that  $\CU=\CU_\phi$.
\end{enumerate} 
\end{thm}

\begin{proof}
Using \cite[Proposition~6.3]{Ric}, the equivalence $(1) \dimp (2)$ follows from Proposition~\ref{4.7}  and Corollary~\ref{corimportante}. Again Proposition~\ref{4.7}  provides $(2) \dimp (3) \dimp (4)$.

$(5) \imp (2)$ is obvious. In order to prove $(2) \imp (5)$ we only should realize that $\CU=\CU_\phi$ where $\phi$ is the sp-filtration $\phi=\phi_\CU$. Trivially $\CU_{\phi}\subset \CU$ because $\phi=\phi_\CU$.
Moreover, under hypothesis $(2)$ for $\CU$, Proposition~\ref{4.7}  shows that $\CU_{\phi}^{\,\,\perp}=\CU^\perp$ that implies $\CU_{\phi}=\CU$.
\end{proof}
 
Let $\Ais{(R)}$ be the class of aisles of $\D(R)$ and $\Filsup(R)$ be the set of all sp-filtrations of $\spec(R)$. Let us denote by $\Aisc{(R)}$ the compactly generated aisles of $\D(R)$. Let us consider on $\Ais{(R)}$ the usual inclusion relation. The order on $\Filsup(R)$ is the induced order by the usual one on $\CP(\spec(R))$. We define a couple of order preserving maps
\[
\Ais{(R)}
\underset{\aaa}{\overset{\fff}{\rightleftarrows}}
\Filsup(R)
\]
by setting $\aaa(\phi):=\CU_\phi$ for any  $\phi\in \Filsup(R)$, and $\fff(\CU):=\phi_\CU$ for any  aisle $\CU$ of $\D(R)$ (see \ref{filtrationbyporaisle}).

\begin{thm}
\label{classification-c-g}
The maps $\fff$ and $\aaa$ establish a bijective correspondence between $\Aisc{(R)}$ and $\Filsup(R)$. 
Furthermore, given $\phi\in \Filsup(R)$ the corresponding t-structure $(\CU_\phi,{\CU_\phi}^\perp[1])$ is described in terms of the sp-filtration by:
  \begin{align*}
         \CU_\phi
         & = \{\,X\in\D(R)\,;\,\,\,
         \supp(\h^j(X))\subset\phi(j), \text{ for all } j\in\ZZ\,\}\\
         {\CU_\phi}^\perp
         & =\{\,Y\in\D(R)\,;\,\,\,\,
         \R\Gamma_{\phi (j)}Y\in\D^{>j}(R), \text{ for all } j\in\ZZ\,\}
  \end{align*}
\end{thm}

\begin{proof}
In the proof of $(2) \imp (5)$ in Theorem~\ref{4.9} we have shown that $\aaa {\circ}\fff(\CU)=\CU$ for any $\CU\in \Aisc{(R)}$.

Conversely, if $\phi\in\Filsup(R)$ let us prove that $\phi =\fff {\circ} \aaa(\phi)$.
%
Let
  \[
   \CU:=\{X\in\D(R)\,;\,\,\supp(\h^j(X ))\subset\phi
   (j)\text{ for all }j\in\ZZ\}.
  \]
The class $\CU$ is a cocomplete pre-aisle of $\D(R)$ which contains $\{R/\ip [-j]\,\,;\,\, j\in\ZZ\text{ and }\ip\in\phi (j)\}$ and, hence, it also contains $\CU_\phi=\aaa (\phi)$. If $X\in\CU$ then the proof of the equivalence $(3) \dimp (1)$ in Theorem \ref{4.7} shows that $\tleq{n}{}X$ belongs to $\CU_\phi$, for every $n\in\ZZ$. So in the canonical distinguished triangle
  \[
   \oplus_{n\geq 0}\tau^{\leq n}X\lto
   \oplus_{n\geq 0}\tau^{\leq n}X\lto
   X\stackrel{+}{\lto},
  \]
the two left vertices are objects in $\CU_\phi$, as a consequence $X\in\CU_\phi$. Therefore $\CU_\phi=\CU$, and from this identification it is easy to derive that $\phi =\fff(\CU_\phi)$.

To conclude let us recall from Proposition~\ref{truncacionparaunideal} that
  \[
  ({\cu{i}{\phi(i)}})^\perp=
  \{Y\in\D(R)\,;\,\,\R\Gamma_{\phi (i)}Y\in\D^{>i}(R)\}\qquad (\forall i\in\ZZ). \]
Hence the displayed description for $\CU_\phi^{\,\,\perp}$ in the statement of the current theorem follows from the obvious relation 
$\CU_\phi^{\,\,\perp}=\bigcap_{i\in\ZZ}({\cu{i}{\phi(i)}})^\perp{}$.
\end{proof}

\begin{rem}
\label{ex-Ne-Ma}
The existence of aisles of $\D(R)$ which are not compactly generated is well-known, as it can be easily derived from  \cite[Theorem 3.3]{Nct}. But unlike the situation in {\lc} residue fields are not the right objects to classify compactly generated $t$-structures. If $R$ is an integral domain (commutative and Noetherian) and not a field, then the aisle $\CU$ of $\D(R)$ generated by $\{k(\ip)\,\,;\,\,\ip\in \spec(R)\}$ is not compactly generated. Indeed, otherwise its associated sp-filtration would be given by $\phi (i)=\spec(R)$, for $i\leq 0$, and $\phi (i)=\varnothing$, for $i>0$ (see Theorem~\ref{classification-c-g}). Then we would have $\CU=\Dleq{0}(R)$. But that is impossible because $R[0]\in\CU^\perp$ since $\Hom_R(k(\ip),R)=0$, for every $\ip\in\spec(R)$.
\end{rem}

\begin{cor}
\label{islasinducidas}
Let $\sharp \in \{-, +, \bb, \text{``blank''}\}$. Let $\CV$ be an aisle (or more generally, any total pre-aisle) of $\Dc^\sharp(R)$ generated by bounded above complexes, and let $\CE$ be its right orthogonal in $\Dc^\sharp(R)$. Then there exists a unique $\phi \in \Filsup(R)$ such that $\CV=\CU_\phi\cap\Dc^\sharp(R)$ and $\CE=\CF_\phi\cap\Dc^\sharp(R)$.
\end{cor}

\begin{proof}
It follows from Proposition~\ref{nuevo} and Proposition~\ref{4.7}.  The uniqueness of $\phi$ follows from the fact $\CV=\CU_\phi\cap\Dc^\sharp(R)$ determines $\phi$ (by Theorem~\ref{classification-c-g} above).
\end{proof}


\section{The weak Cousin condition}
\label{wccondition}

We proceed to classify, under sufficiently general
hypotheses, on $R$, all the compactly generated $t$-structures on $\D(R)$ that restrict to $t$-structures on $\Dc^\sharp(R)$. 
The main result of this section is Theorem~\ref{5.5} ---it provides a necessary condition on a sp-filtration $\phi$ in order to $\CU_\phi\cap\Dc^{\sharp}(R)$ be
an aisle of $\Dc^{\sharp}(R)$. 
Note that the statement of Theorem~\ref{5.5} here and \cite[Proposition 7.4]{Stanley} are almost the same. Here we treat with the case of the unbounded category $\Dc(R)$ and obtain in particular the result for $\Dbc(R)$ (the framework in {\lc}).

\begin{cosa}
\label{notation}
Given a sp-filtration $\phi\colon \ZZ\lto\CP(\spec(R))$ we will denote by $\tleq{ }{\phi}$ the left truncation functor associated to the aisle $\CU_\phi$ and by $\tg{ }{\phi}$ the right truncation functor. So that for each $M\in\D(R)$ the diagram
  \[
   \tleq{}{\phi}M\lto M
   \lto \tg{}{\phi}M\stackrel{+}{\lto}
  \]
denotes the natural distinguished triangle determined by the $t$-structure $(\CU_\phi, \CU_\phi^\perp [1])$ for $\Mdot$; we refer to this triangle as the $\phi$-\emph{triangle with central vertex} $\Mdot\in\D(R)$ or just a $\phi$-\emph{triangle}.




The assumption that $\CU_\phi\cap\Dc^\sharp(R)$ is an aisle of  $\Dc^\sharp(R)$ is equivalent to the fact that the $\phi$-triangle in $\D(R)$ with central vertex $\Xdot$ belongs to $\Dc^\sharp(R)$ whenever $\Xdot\in\Dc^\sharp(R)$. In other words, $\CU_\phi\cap\Dc^\sharp(R)$ is an aisle of $\Dc^\sharp(R)$ if and only if $\tleq{}{\phi}\Xdot\in \Dc^\sharp(R)$ (or equivalently $\tg{}{\phi}\Xdot\in \Dc^\sharp(R)$) for all $\Xdot\in\Dc^\sharp(R)$.



\end{cosa}

The following lemmas are useful in the proof of Theorem~\ref{5.5}.

\begin{lem}
\label{5.3}
Let $\phi\colon\ZZ\to\CP(\spec(R))$ be a sp-filtration. Then, for every $j\in\ZZ$,  we get $\tleq{}{\phi} \Dgeq{j}(R)\subset \Dgeq{j}(R)$ and $\tg{}{\phi} \Dgeq{j}(R)\subset \Dgeq{j}(R)$.
\end{lem}

\begin{proof}
Without loss of generality, we may assume that $j=0$. Let $X\in\D^{\geq 0}(R)$ be any complex and put 
$\Tdot =\tleq{}{\phi} \Xdot$ and $\Ydot=\tg{}{\phi} X$. From the long exact sequence of
homology
associated to the canonical
$\phi$-triangle with central vertex $\Xdot$ we obtain isomorphisms $\h^{i-1}(Y )\cong \h^i(T)$ for
all $i<0$, and a monomorphism of $R$-modules $\h^{-1}(\Ydot)\inc \h^0(\Tdot)$. In particular $\supp (\h^{j}(\Ydot))\subset \supp (\h^{j+1}(\Tdot))\subset \phi(j+1)\subset \phi(j)$ for all $j\leq -1$. The explicit description of $\CU_\phi$ given in Theorem~\ref{classification-c-g} shows that
$\tleq{-1}{}\Ydot \in \CU_\phi$. So having in mind that $Y \in \CU_\phi^\perp$  the canonical map
$\tleq{-1}{ }Y \to Y$ is zero.  Thus $Y\in\Dgeq{0}(R)$ and as a consequence $T\in\Dgeq{0}(R)$.
\end{proof}

As usual, we denote by $\ass(M)$ the set of associated prime ideals of a module $M\in\Mod(R)$ (for the basic properties of associated prime ideals {\cfr} \cite[\S 6, page 38]{ma}).

\begin{lem}
\label{5.4}
Let $\phi$ be a sp-filtration of $\spec(R)$. Let $j$ be an integer and $M$ be a finitely generated $R$-module such that $\ass  (M)\cap\phi (j)=\varnothing$ (\emph{e.g.} $M={R}/{\ip }$, with $\ip \not\in\phi (j)$). Let 
  \[
   T\stackrel{a}{\lto} M
   [-j]\stackrel{b}{\lto} Y\stackrel{+}{\lto}
  \]
be the canonical $\phi$-triangle with central vertex $M[-j]\in\D(R)$. Then:
  \begin{enumerate}
    \item $T\in \Dg{j}(R)$;
    \item $Y\in \Dgeq{j}(R)$ and $\Gamma_{\phi(j)}(\h^j(\Ydot))=0$; and
    \item the homomorphism of $R$-modules
          $\h^j(b) \colon  M \to \h^j(Y)$ is
          an essential extension.
  \end{enumerate}
\end{lem}

\begin{proof}
Assuming again that $j=0$ and rewriting the proof of Lemma~\ref{5.3} for $\Xdot = M[0]=M$, we get that  $Y$  and $T$ belong to $\Dgeq{0}(R)$ and that the canonical map $\h^0(a)\colon \h^0(T)\to M$ is a monomorphism of $R$-modules.  Then $\ass(\h^0(T))\subset\ass (M)$. The hypothesis on $\ass(M)$ implies that  $\ass (\h^0(T))=\varnothing,$ then
$\h^0(T)=0$ and $T\in \D^{>0}(R)$.

Having in mind that $\h^0(Y )[0]\cong\tleq{0}{}Y,$ we get that
   \begin{align*}
    \Hom_{R}(N, \h^0(Y ))
    & = \Hom_{\D(R)}(N[0], \tleq{0}{}Y)\\
    & \cong \Hom_{\D(R)}(N[0], Y)=0
   \end{align*}
for every $R$-module $N$ such that $\supp (N)\subset \phi (0)$; therefore $\Gamma_{\phi(0)}(\h^0(\Ydot))=0$. Finally, let us check that the monomorphism $\iota:=\h^0(b) \colon M\inc \h^0(Y)$ is essential. Consider the exact sequence 
  \[
   0\rightarrow M \stackrel{\iota}{\lto}
   \h^0(Y) \stackrel{\nu}{\lto}
   \h^1(T)\rightarrow 0
  \]
of $R$-modules associated to the $\phi$-triangle with central vertex $M=M[0]$. Let  $V\subset \h^0(Y )$  be a finitely generated submodule such that $\Img(\iota )\cap V=0.$ Then $\Ker(\nu)\cap V= \Img(\iota )\cap V=0$, so the
composition $V\hookrightarrow \h^0(Y)\stackrel{\nu}{\lto} \h^1(T)$ is also a monomorphism, hence $\supp (V)\subset \supp (\h^1(T))\subset\phi (1)\subset\phi (0)$. Therefore $V=\Gamma_{\phi(0)}(V)\subset \Gamma_{\phi(0)}(\h^0(\Ydot))=0$.
\end{proof}

\begin{thm}
\label{5.5}
Let $\phi \colon \ZZ\to\CP(\spec(R))$ be a sp-filtration. Suppose that $\ip \subsetneq\iq $ is a strict inclusion of prime ideals of $R$ such that $\ip $ is maximal under $\iq $. Let
  \[
   T\lto{R}/{\ip }[-j+1]\lto
   Y\stackrel{+}{\lto}
  \]
be a $\phi$-triangle in $\D(R)$ with central vertex ${R}/{\ip }[-j+1]$.
If $\iq \in\phi (j)$ and $\ip \not\in\phi (j-1)$, then neither $T$ nor $Y$ belongs to $\Dc(R).$
\end{thm}

\begin{proof}
For simplicity assume $j=1$, and put $\CU=\CU_\phi$. Suppose that one of the objects $T$ or $Y$ belongs to $\Dc(R)$, then the other belongs as well. Hence  $T\to{R}/{\ip }[0]\to Y\stackrel{+}{\to}$ is a triangle in $\Dc(R)$ with $T\in\CU$ and $Y\in\CU^\perp$. By Proposition~\ref{localizacion} localizing at $\iq $, we get a triangle
  \[
   T_\iq \lto {R_\iq }/{\ip R_\iq }[0]\lto
   Y_\iq \stackrel{+}{\lto}
  \]
in $\Dc(R_\iq),$ such that $T_\iq \in\CU_\iq $ and $Y_\iq \in{\CU_\iq}^\perp.$ Moreover, from Proposition~\ref{localizacion} we know that the sp-filtration $\phi_\iq$ of $\spec(R_\iq )$ associated to $\CU_\iq $ is given by $\phi_\iq (i)=\phi (i)\cap \spec (R_\iq )$. As a consequence $\iq R_\iq \in\phi_\iq (1)$ but $\ip R_\iq \not\in\phi_\iq (0)$. Let us simplify the notation assuming that $R=R_\iq$ is local, $\iq =\im$ is the maximal ideal of $R$, and $\ip$ is maximal under $\im$.

By Lemma~\ref{5.4}, under the present hypothesis $T\in \Dg{0}(R)$, $Y\in \Dgeq{0}(R)$, $\Gamma_\im(\h^0(Y))=0$ and the induced homomorphism  ${R}/{\ip}\to \h^0(Y)$ is an essential extension. Then $\supp(\h^0(Y))=\V(\ip)$, and hence $\supp (\h^1(T))\subset \V(\ip)\cap\phi(1)=\{\im\}$.  Therefore $\h^1(T)$ is a finitely generated $R$-module with $\supp (\h^1(T))\subset \{\im\}$, so one can find $r\in \NN$ such that $\im^r \h^1(T)=0.$

Let us fix an integer $k>0$. Let $j\colon R\to R/\ip$ be the canonical homomorphism of rings. The ring $A:=R/\ip$  is an integral local domain of Krull dimension 1 with ${\idealn}:=\im/\ip$ as its maximal ideal.
Notice that the canonical map
  \[
   \ext_{A }^1({A}/{{\idealn}^k},A)
   \stackrel{\alpha}{\lto}
   \ext_{R}^1({R}/({\ip +\im  ^k}),{R}/{\ip })
  \]
is injective. Moreover, applying the homological functor $\Hom_{\D(R)}(-,R/\ip)$ to the short exact sequence of $R$-modules
  \[
   0\rightarrow ({\ip +\im ^k})/{\im ^k}
   \lto {R}/{\im ^k}\lto
   {R}/({\ip +\im ^k})\rightarrow 0,
  \]
we obtain an exact sequence
  \[
   0 \lto
   \ext_R^1(R/(\ip +\im ^k) ,R/\ip )
   \stackrel{\beta}{\lto}\ext_R^1(R/\im ^k ,R/\ip )
  \]
because $\Hom_R(({\ip +\im^k})/{\im^k},{R}/{\ip})=0$. Therefore we get a monomorphism $\beta \alpha \colon \ext_{A }^1({A}/{{\idealn}
^k},A)\inc \ext_R^1(R/\im^k, R/\ip )$. Now
\[
\ext_{A }^1({A}/{{\idealn}^k},A) \cong
   \Hom_{A}({A}/{{\idealn}^k},{Q(A)}/{A}),\]
because $Q(A)=k(\ip )$, the field of quotients of $A=R/\ip$, is the injective hull of $A$ in $\Mod(A)$. Note that this hom can be described as the $A$-submodule of ${Q(A)}/{A}$ of those elements $\bar{x}=x+A\in Q(A)/A$ such that ${\idealn}^k \bar{x}=0$. Since for each $k > 0$ one can always find elements $\bar{x}\in Q(A)/A$ such that ${\idealn}^{k-1}\bar{x}\neq 0={\idealn}^{k}\bar{x}$,  we conclude from the existence of the monomorphism  $\beta \alpha$ that 
 \begin{equation}\label{resultado}
    \im  ^{k-1}\ext^1_R({R}/{\im  ^k},{R}/{\ip })
    \neq 0,\qquad \forall k>0
 \end{equation}

On the other hand, since $\im  \in\phi (1)$ (i.e. ${R}/{\im}[-1]\in\CU$),  by Proposition~\ref{prop-Two-ideals}$(3)$ we get that ${R}/{\im^k}[-1]\in\CU$, for all $k>0$. The properties of the triangle in the hypothesis of the theorem give us isomorphisms 
  \begin{align*}
         \ext_R^1({R}/{\im  ^k},{R}/{\ip })
         &\cong \Hom_{\D(R)}({R}/{\im  ^k}[-1],{R}/{\ip })\\
         &\cong \Hom_{\D(R)}({R}/{\im  ^k}[-1],T)
 \end{align*}
But,
since $T\in \Dg{0}(R)$, we have that
  \[
   \Hom_{\D(R)}({R}/{\im  ^k}[-1],T)\cong
   \Hom_R({R}/{\im  ^k},\h^1(T)),
  \]
so $\Hom_{\D(R)}({R}/{\im  ^k}[-1],T)$ is isomorphic to a submodule of $\h^1(T)$. Hence $\im ^r\ext_R^1({R}/{\im  ^k},{R}/{\ip })=0$, for all $k>0$. This fact contradicts (\ref{resultado}).
\end{proof}

\begin{rem} 
A dualizing complex can be explicitly realized as a residual complex and determines a codimension functor (see remark \ref{interpreting-codimension-function} further on). The codimension function  provides a sp-filtration $\fcm\colon\ZZ\to \CP(\spec(R))$ that satisfies the following condition:
\begin{enumerate}
         \item[]For any $j\in\ZZ$, and any pair of prime ideals $\ip \subsetneq\iq$,
         with $\ip $ maximal under $\iq$, then
         $\iq \in\fcm (j)$ if and only if $\ip \in\fcm (j-1).$
  \end{enumerate}
We call this property the \emph{strong Cousin condition}.  For our purposes it is convenient to consider sp-filtrations under a weaker version of the above condition ({\cfr} Theorem~\ref{5.5}). This fact justifies the following.
\end{rem}

\begin{defn}
Let $\phi\colon\ZZ\to\CP(\spec(R))$ be a sp-filtration. We say that $\phi$ satisfies the \emph{weak Cousin condition}  if the following property holds:
\begin{enumerate}
         \item[]For every $j\in\ZZ,$ if $\ip \subsetneq\iq $
         are prime ideals, with $\ip $ maximal under $\iq$,
         and $\iq \in\phi (j)$ then $\ip \in\phi (j-1).$
  \end{enumerate}
\end{defn}

\begin{cor}
\label{Cousin}
Let $\sharp \in \{-,+,\bb, \text{``blank''}\}$. If $\phi$ is a sp-filtration of $\spec(R)$ such that $\CU_\phi \cap \Dc^{\sharp}(R)$ is an aisle of $\Dc^{\sharp}(R)$, then $\phi$ satisfies the  weak Cousin condition.
\end{cor}
\begin{proof}
Straightforward consequence of Theorem~\ref{5.5}.
\end{proof}

Our next goal is to see whether the converse of the statement in Corollary~\ref{Cousin} is also true. For that we study the sp-filtrations satisfying  the weak Cousin condition.

\begin{cosa}
\label{sobrecomponentes}
Recall that for two prime ideals $\ip,\,\iq\in\spec(R)$ the relation $\ip \subset\iq $ can be expressed saying that $\ip$ is a generalization of $\iq$ or, equivalently, $\iq$ is a specialization of $\ip.$ A subset $Y\subset \spec(R)$ is stable under generalization if $\ip \in Y$ whenever  $\ip \subset\iq $ with $\iq\in Y$.
For instance if $\iq\in\spec(R)$ we identify
$\spec(R_\iq)$ with the subset of all generalizations of $\iq$ in $\spec(R)$.

Under the assumption that $R$ is a Noetherian ring, a subset $Y\subset \spec(R)$ is stable under specialization (sp-subset) and generalization if and only if $Y$ is open and closed, equivalently  $Y$ is the union of connected components of $\spec(R)$. Indeed, let $Y\subset \spec(R)$ be stable under specialization and generalization. If $\ip\in\spec(R)$ is a minimal prime ideal such that $\V(\ip)\cap Y\neq \varnothing$ necessarily $\ip\in Y$ because $Y$ is stable under generalization; thus $\V(\ip)\subset Y$ since $Y$ is also stable under specialization. Let $\Min(R)=\{\ip _1, \dots, \ip _s\}$ be the set of minimal prime ideals of $R$ order in such a way that $\Min(R)\cap Y=\{\ip _1, \dots, \ip _r\}$, for an integer $r\leq s$. Then $Y=\cup_{i=1}^r\V(\ip_i)$ so it is closed, and it is open because $\spec(R)\setminus Y=\cup_{i=r+1}^s\V(\ip_{i})$.
\end{cosa}

\begin{prop}
\label{discreteness-of-filtration}
Let   $\phi \colon \ZZ\to\CP(\spec(R))$ be a sp-filtration  that satisfies the  weak Cousin condition. Then there exists an integer $j_0\in \ZZ$ such that $\phi(j)=\phi(j_0)$ for all $j\leq j_0$, and the subset $\phi(j_0)\subset \spec(R)$ is open and closed.
Also, the set $\bigcap_{i\in\ZZ}\phi (i)$ is open and closed.
\end{prop}

\begin{proof} The class of sp-subsets of $\spec(R)$ is closed under taking arbitrary unions and intersections, from which we get that $\bigcap_{i\in\ZZ}\phi (i)$ and $\bigcup_{i\in\ZZ}\phi (i)$ are sp-subsets. Furthermore  the  weak Cousin condition  implies that the sets $\bigcap_{i\in\ZZ}\phi (i)$ and $\bigcup_{i\in\ZZ}\phi (i)$ are both also stable under generalization, so they are at once open and closed in $\spec(R)$ (see \ref{sobrecomponentes}). The set of  minimal prime ideals of $Y=\bigcup_{i\in\ZZ}\phi (i)$ is finite, so we can find a small enough integer $j_0$ such that $\phi(j_0)$ contains all minimal prime ideals of $Y$ since the sp-filtration is decreasing. Then $\phi(j_0)=Y$,  and  $\phi(j_0)=\phi (j)$ for all $j\leq j_0$.
\end{proof}

\begin{cor}
\label{discreteness-of-filtration2}
If $\spec(R)$ is connected and $\phi$ is not one of the two trivial constant sp-filtrations, then the following assertions hold:
 \begin{enumerate}
    \item The sp-filtration $\phi$ is
          \emph{separated}, \emph{i.e.} $\bigcap_{i\in\ZZ}\phi (i)=
          \varnothing$.
    \item There exists an integer $j_0\in\ZZ$ such that
          $\phi(j_0)=\spec(R)$.
    \item If $R$ has finite Krull dimension, then there exists a large enough $k\in\ZZ$ such that $\phi(k)=\varnothing$.
 \end{enumerate}
\end{cor}

\begin{proof}
Being $\spec(R)$ connected and $\phi$ not one of the two trivial constant sp-filtrations it follows that $\bigcap_{i\in\ZZ}\phi (i)=\varnothing$ and $\bigcup_{i\in\ZZ}\phi (i)=\phi(j_0)=\spec(R)$ with $j_0$ as in the previous Proposition. 

To prove (3), denote by $d$ the Krull dimension of $R$ and by $\Min(R)$ the set of minimal prime ideals of $\spec(R)$. By $(1)$ we can associate to each maximal $\im\subset R$ an integer
$i_\im=\max\{i\in\ZZ\,\,;\,\,\im\in\phi(i)\}$. Let us fix a maximal ideal $\im\subset R$ and take a maximal chain of prime
ideals $\ip_0\subsetneq\ip_1\subsetneq\dots\subsetneq\ip_{r}=\im$ (then $r\leq d$). The weak Cousin condition  for $\phi$ implies that the minimal
prime ideal $\ip_0$ belongs to $\phi
(i_\im-r)\subset\phi (i_\im-d)$. In particular, we
have $\Min(R)\cap\phi (i_\im-d)\neq\varnothing$, which
implies that $i_\im-d\leq m:=\max\{i\in\ZZ\,\,;\,\,\Min(R)\cap\phi (i)\neq\varnothing\}$. Then $i_\im\leq
d+m$, for every maximal $\im\subset R$. So $\phi (d+m+1)$ does not
contain any maximal ideal, which means that $\phi
(d+m+1)=\varnothing$.
\end{proof}

\begin{rem}
For a general sp-filtration $\phi$ the result in Lemma~\ref{5.3} shows that $\CU_\phi\cap\Dmas{}(R)$ is an aisle of $\Dmas{}(R)$. If furthermore there is an integer $k$ such that $\phi (k)=\varnothing$ then $\CU_\phi\subset\Dl{k}{}(R)$. Therefore $\tleq{}{\phi}\Dmenos{}(R)\subset \Dmenos{}(R)$ (equivalently $\tg{}{\phi}\Dmenos{}(R)\subset \Dmenos{}(R)$). In this case $\CU_\phi\cap\D^\sharp(R)$ is an aisle of $\D^\sharp(R)$ for all $\sharp\in\{-,+,\bb\}$.
\end{rem}

As a consequence of the above Proposition we get the following. 

\begin{cor}
\label{acotacion-wcc-fil}
Assume that  $\spec(R)$ is connected, $R$ has finite Krull dimension, and that $\phi$ is a non-constant sp-filtration of $\spec(R)$ satisfying the  weak Cousin condition. For any superscript  $\sharp\in\{-,+,\bb\}$, $\CU_\phi\cap\D^\sharp(R)$ is an aisle of $\D^\sharp(R)$. Moreover, there exist integers $j\leq k$ for which the canonical maps $\tl{k}{}X \to X$ and $X\to\tg{j}{ }X$ induce isomorphisms $\tleq{}{\phi}\tl{k}{}\Xdot\iso\tleq{}{\phi}\Xdot$  and $\tg{}{\phi}\Xdot\iso \tg{}{\phi}\tg{j}{ }\Xdot$, for all $\Xdot\in\D(R)$.
\end{cor}

\begin{proof} By Corollary~\ref{discreteness-of-filtration2} there exist integers $j\leq k$ such that $\phi (j)=\spec(R)$ and $\phi (k)=\varnothing$. So the remark preceding this corollary shows that $\CU_\phi\cap\D^\sharp(R)$ is an aisle of $\D^\sharp(R)$.

The canonical homomorphism $X {\to}\,\tg{j}{}X$ induces natural isomorphisms
  \[
   \Hom_{\D(R)}(\tg{j}{}\Xdot,\Ydot)\iso \Hom_{\D(R)}(\Xdot,\Ydot)
  \]
for all $\Ydot\in\CU_\phi^{\,\perp}$ because $\Dleq{j}{}(R)\subset \CU_\phi$. As a consequence of the natural adjunction isomorphisms $\Hom_{\D(R)}(\Ndot,\Ydot)\iso   \Hom_{\D(R)}(\tg{ }{\phi}\Ndot,\Ydot)$ for any $\Ndot\in\D(R)$ and $\Ydot\in\CU_\phi^{\,\perp}$, we get that the natural map $\tg{}{\phi}X {\to}\,\tg{}{\phi}\tg{j}{}X$ is an isomorphism.

Following a dual path, note that the canonical map $\tl{k}{}X {\to}X$ induces, for all $U\in\CU_\phi$, natural isomorphisms
  \[
   \Hom_{\D(R)}(U,\tl{k}{}X)\iso \Hom_{\D(R)}(U,\Xdot)
  \]
because $ \CU_\phi\subset \Dl{k}{}(R)$. Then the adjunction isomorphism
  \[
   \Hom_{\D(R)}(U,\Ndot)\iso   \Hom_{\D(R)}(U,\tleq{}{\phi}\Ndot),
  \]
for any $\Ndot\in\D(R)$ and $U\in\CU_\phi$, lead us to conclude that the natural map $\tleq{}{\phi}\tl{k}{} X {\to}\,\tleq{}{\phi} X $ is an isomorphism. 
\end{proof}

\begin{cor}
\label{5.11}
Let  $\phi$ be a sp-filtration of $\spec(R)$. Consider the following assertions:
  \begin{center}
     $(\sharp)\quad \CU_\phi\cap\Dc^\sharp(R)$
     is an aisle of $\Dc^\sharp(R)$,
  \end{center}
with $\sharp\in\{-,+,\bb,\text{``blank''}\}$.
Then $(-) \imp (\bb)$ and, if $R$ has finite Krull dimension then the assertions $(-)$, $(+)$, $(\bb)$ and $(\text{``blank''})$ are equivalent.
\end{cor}

\begin{proof} First of all note that under any of the assumptions labeled $(\sharp)$ the sp-filtration $\phi$ satisfies the  weak Cousin condition  ({\cfr} Corollary~\ref{Cousin}). Furthermore we can assume that $\spec(R)$ is
connected and $\phi$ is not one of the two trivial constant sp-filtrations.

We can rewrite each of the four statements here by saying that $\tleq{}{\phi}\Xdot\in \Dc^\sharp(R)$ (equivalently $\tg{}{\phi}\Xdot\in \Dc^\sharp(R)$) for all $\Xdot\in\Dc^\sharp(R)$.

First let us show that $(-)\imp (\bb)$. If $X\in\Dbc(R)$ then the $\phi$-triangle with central term $X$ is in $\Dmas{}(R)$ (see Lemma~\ref{5.3}). By assumption ($-$) we also have that it lays on $\Dcmenos(R)$, hence on $\Dbc(R)$.

In the rest of the proof we suppose in addition that $R$ has finite Krull dimension. Then Corollary~\ref{acotacion-wcc-fil} applies here. Let $j\leq k$ be the integers in the statement of Corollary~\ref{acotacion-wcc-fil}.
In order to prove $(\bb)\imp (+)$ let us take any complex $X\in\Dcmas(R)$; then  we get that $\tleq{}{\phi}X\cong \tleq{}{\phi}\tl{k}{}X \in\Dbc(R)\subset \Dcmas(R)$ by $(\bb)$ because $\tl{k}{}X\in\Dbc(R)$. To check $(+)\imp (\text{``blank''})$ notice that for every $X\in\D(R)$ it holds that $\tg{}{\phi}X\cong \tg{}{\phi}\tg{j}{}X$ by Corollary~\ref{acotacion-wcc-fil}; then $\tg{}{\phi}X\cong \tg{}{\phi}\tg{j}{}X \in\Dcmas(R)\subset \Dc{}(R)$ by $(+)$. Finally $(\text{``blank''})\imp (-)$ is a consequence of the fact that $\CU_\phi\subset\Dleq{k}{}(R)$.
\end{proof}

\begin{cor}
\label{bounded=right-bounded}
Over a commutative Noetherian ring of finite Krull dimension the problems of classifying t-structures on $\Dcmenos(R)$ and $\Dbc(R)$ are equivalent. They are also equivalent to classifying on $\Dc{}(R)$ and $\Dcmas(R)$ all t-structures  generated by perfect complexes (or, by bounded above complexes).
\end{cor}

\begin{proof}
It follows from Corollary~\ref{islasinducidas} and Corollary~\ref{5.11}.
\end{proof}

\begin{cosa}
\label{idempotentes}
The sp-filtrations of $\spec(R)$ corresponding to Bousfield localizations on $\D(R)$ are exactly the constant sp-filtrations. If $\phi$ is a  constant sp-filtration inducing a Bousfield localization on $\Dc^\sharp(R)$ then $\phi$ satisfies weak Cousin condition (for any $\sharp \in \{-,+,\bb,\text{``blank''}\}$). By Proposition~\ref{discreteness-of-filtration}, there exists a subset $Z\subset\spec(R)$ open and closed such that $\phi(j)=Z$ for all $j\in\ZZ$. Hence $Z=\V(e)$ for an idempotent element $e\in R$. For each complex $\Xdot \in \D(R)$ the Bousfield triangle associated to $Z$ is 
\[
\R\Gamma_{\V(e)}X \lto X \lto 
   X_e \stackrel{+}{\lto}
\] 
Note that $X\cong\R \Gamma_{\V(e)}(X) \oplus X_e$, because the third map  in the triangle is zero.  So trivially the Bousfield localization $(\CU_\phi, \CU_\phi^\perp[1])$ restricts to a Bousfield localization on $\Dc^\sharp(R)$. These are all the Bousfield localizations on $\Dc^\sharp(R)$. In particular, if $\spec(R)$ is connected then the Bousfield localizations generated by bounded above complexes (equivalently, the Bousfield localizations generated by perfect complexes) on $\Dc^\sharp(R)$ are just the trivial ones.
\end{cosa}

%


\section{
Aisles determined by finite filtrations by supports
}
\label{descripciontruncados}

In this section we set the stage to show that, under sufficiently general hypothesis, the converse of the statement in Corollary~\ref{Cousin} is true.

\begin{cosa}
\label{cor-lema}
If $R$ has finite Krull dimension and $\spec(R)$ is connected, for each nonconstant sp-filtration $\phi\colon \ZZ\to\CP(\spec(R))$ satisfying the  weak Cousin condition  there exist integers $j_0\leq k$ such that $\phi(j_0)=\spec(R)$ and $\phi(k)=\varnothing$ ({\cfr} Corollary~\ref{discreteness-of-filtration2}). For a general commutative Noetherian ring $R$ we introduce the following definition.

\begin{defn}
Let $\phi\colon \ZZ\to\CP(\spec(R))$ be a sp-filtration. Let $s \leq n$ two integers. We say that $\phi$ is \emph{determined in the interval} $[s,n]$ if $\phi(j)=\phi(s)$ for all $j\leq s$, $\phi(s) \supsetneq \phi(s+1)$ and $\phi(n) \supsetneq\phi(n+1)=\varnothing$.

If the sp-filtration $\phi\colon \ZZ\to\CP(\spec(R))$ is determined in the interval $[s,n]\subset\ZZ$ we say that $\phi$ is \emph{finite of length} $\lth(\phi):=n-s+1$. For any finite sp-filtration $\phi$ we have that $\lth(\phi)\geq 1.$ 
\end{defn}

\begin{rem} Although a constant sp-filtration is never determined in an interval in the above sense, we adopt the convention that a constant sp-filtration $\phi$ of $\spec(R)$ is finite of length $\lth(\phi):=0$. The associated aisle to a constant sp-filtration is a localizing class corresponding to a $\otimes$-compatible Bousfield localization ({\cfr} \S \ref{util0}).
\end{rem}

\end{cosa}

\begin{cosa}
\label{parausar}
Let $\phi\colon \ZZ\to \CP (\spec (R))$ be a finite sp-filtration of length $1$, that is a sp-filtration such that $\CU_{\phi}=\cu{i}{Z}$ for a fixed $i\in\ZZ$ and $Z=\phi(i)$. Corollary~\ref{ejemplo} shows that its associated left truncation functor $\tleq{ }{\phi}{ }$ is $\tleq{i}{}\R\Gamma_{{\phi}(i)}$. Note that for each complex $X\in\D(R)$, $\tg{ }{\phi}{ }X$ is determined by the existence of a distinguished triangle in $\D(R)$
  \[
   \tg{i}{}{\R\Gamma_{Z}} X {\lto}\,\, \tg{ }{\phi}{ }X \lto 
   {\R Q_Z} X \overset{+}{\lto}
  \]
built as follows.
The natural map ${\pi}\colon\tleq{ }{\phi}{ }X {\to}X$ is the composition of the canonical maps ${\alpha}\colon \tleq{i}{}{\R\Gamma_{Z}}X{\to} {\R\Gamma_{Z}}X$ and ${\rho}\colon{\R\Gamma_{Z}}X{\to} X.$ Applying the octahedron axiom to the commutative diagram $\pi=\rho{} {\circ}\alpha$
  \[
   \begin{tikzpicture}[scale=0.9]
      \draw[white] (0cm,2cm) -- +(0: \linewidth)
      node (E) [black, pos = 0.45] {$\Xdot$}
      node (H) [black, pos = 0.92] {}
      node (F) [black, pos = 0.8] {${\R Q_Z} \Xdot$};
      \draw[white] (0cm,1.675cm) -- +(0: \linewidth)
      node (Z) [black, pos = 0.1] {}
      node (L) [black, pos = 0.9] {};
      \draw[white] (0cm,0.7cm) -- +(0: \linewidth)
      node (A) [black, pos = 0.16] {$\tleq{i}{}{\R\Gamma_{Z}} \Xdot$};
      \draw[white] (0cm,3.05cm) -- +(0: \linewidth)
      node (W) [black, pos = 0.1] {}
      node (T) [black, pos = 0.9] {};
      \draw[white] (0cm,3.25cm) -- +(0: \linewidth)
      node (C) [black, pos = 0.37] {$\tg{i}{}{\R\Gamma_{Z}} \Xdot$};
      \draw[white] (0cm,4.2cm) -- +(0: \linewidth)
      node (U) [black, pos = 0.1] {}
      node (V) [black, pos = 0.9] {};
      \node (B) at (intersection of A--C and F--E) {${\R\Gamma_{Z}}\Xdot$};
      \node (K) at (intersection of A--E and C--F) {$\tg{ }{\phi}{ }X$};
      \node (D) at (intersection of A--B and U--V) { };
      \node (I) at (intersection of A--E and W--T) { };
      \node (J) at (intersection of K--F and Z--L) { };
      \draw [->] (A) -- (B) node[left=2.5pt, near end, scale=0.75]{$\alpha$};
      \draw [->] (B) -- (C);
      \draw [->] (C) -- (D) node[above, midway, sloped, scale=0.75]{+};
      \draw [->] (B) -- (E) node[above, midway, scale=0.75]{$\rho$};
      \draw [->] (E) -- (F);
      \draw [->] (K) -- (I) node[above, midway, sloped, scale=0.75]{+};
      \draw [->] (C) -- (K);
      \draw [->] (K) -- (F);
      \draw [->] (E) -- (K);
      \draw [->] (F) -- (J) node[below, midway, sloped, scale=0.75]{+};
      \draw [->] (F) -- (H) node[above, midway, sloped, scale=0.75]{+};
      \draw [->] (A) -- (E) node[below, midway, scale=0.75]{$\pi$};
   \end{tikzpicture}
  \]
we get the vertex of the triangle with base ${\pi}\colon\tleq{ }{\phi}{ }X \to X$ inserted in the triangle we are looking for.
\end{cosa}

\begin{cosa}
\label{parausar2}
Our next task is the description of the truncation functors associated to any finite sp-filtration, see Proposition~\ref{Kashiwara=ours} and Corollary~\ref{segundo-cor-lema}. These results are used in the proof of Lemma~\ref{induction-tool} that let us establish an inductive way to achieve the classification of aisles of $\Dc^\sharp(R)$ in the last section of the paper. 

In \S~\ref{parausar} we have described the truncation functors associated to any finite sp-filtration of length $1$. Let $\phi\colon \ZZ\to\CP(\spec(R))$ be any sp-filtration. For each integer $i\in\ZZ$ let us denote by $\phi_i \colon \ZZ\to\CP(\spec(R))$ the sp-filtration determined by $\CU_{\phi_i}:= \cu{i}{\phi(i)}$. The family $\{\phi_i\,;\, i\in\ZZ\}$ of sp-filtrations of length $1$ determines $\phi$.

Starting from a finite sp-filtration $\phi\colon \ZZ\to\CP(\spec(R))$  determined in the interval  $[s,n]\subset \ZZ$,  we construct a finite sp-filtration $\phi'\colon \ZZ\to\CP(\spec(R))$ of length $\lth(\phi')\leq\lth(\phi)$ by setting  $\phi '(n)=\varnothing$ and $\phi '(j)=\phi (j)$, for all $j\leq n-1$ (that is, ${\phi '}_j=\phi_j$ for any $j\leq n-1$). If $\lth(\phi)>1,$ then $\phi '$ is a finite sp-filtration of length $1\leq \lth(\phi')= \lth(\phi)-1$. Note that if $\phi$ satisfies the  weak Cousin condition then so  does $\phi'$.


\end{cosa}

\begin{lem}
\label{componiendo-2-trucados-derecha}
Let us fix two integers $i\leq j$ and $Z_{j}$, $Z_{i}$ two sp-subsets of $\spec(R)$.  Let $\phi_i$ and $\phi_j$ be the sp-filtration of $\spec(R)$ of length $1$ determined by $\CU_{\phi_k} := \CU_{Z_k}^k$ with $k \in \{i, j\}$. 
If $M \in {\CU_{\phi_i}}^{\perp}$ then $\tg{ }{{\phi_j}}{} {\Mdot} \in {\CU_{\phi_i}}^{\perp}$.
\end{lem}

\begin{proof}
By \ref{parausar}  there is a distinguished triangle in $\D(R)$
  \[
   \tg{j}{}\R\Gamma_{{Z_j}} M {\lto}\,\, \tg{ }{{\phi_j}}{} M
   \lto \R Q_{{Z_j}} M \overset{+}{\lto}
  \]
First note that $\R\Gamma_{{Z_i}} \tg{j}{}\R\Gamma_{{Z_j}} M$ belongs to $\Dg{j}(R)$ and as a consequence
  \[
   \tleq{ }{{\phi_i}} {} \tg{j}{}\R\Gamma_{{Z_j}} M =
   \tleq{i}{}\R\Gamma_{{Z_i}} \tg{j}{}\R\Gamma_{{Z_j}}
   M =0,
  \]
that is, $\tg{j}{}\R\Gamma_{{Z_j}} M\in {\CU_{\phi_i}}^{\perp}$. Let us check that $\R Q_{{Z_j}} M \in {\CU_{\phi_i}}^{\perp}$. This follows from the canonical isomorphisms
  \[
   \tleq{ }{{\phi_i}} {} \R Q_{{Z_j}} M =\tleq{i}{} \R\Gamma_{Z_i}
   \R Q_{{Z_j}} M \cong \tleq{i}{} \R Q_{{Z_j}}
   \R\Gamma_{{Z_i}} M.
  \]
Indeed, $M \in {\CU_{\phi_i}}^{\perp}$ therefore $\R\Gamma_{{Z_i}} M \in \Dg{i}(R),$ hence $\R Q_{{Z_j}}\R\Gamma_{{Z_i}} M \in \Dg{i}(R)$, thus
  \(
   \tleq{ }{{\phi_i}} \R Q_{Z_j} M \cong
   \tleq{i}{}\R Q_{Z_j}
   \R\Gamma_{Z_i} M=0 
  \) 
\end{proof}

\begin{prop}
\label{Kashiwara=ours}
Let $\phi:\ZZ\to\CP(\spec(R))$ be a finite sp-filtration determined in the interval $[s,n]\subset \ZZ$.  Then $\tleq{}{\phi} \Xdot \in\Dleq{n}(R)$ for all $X \in \D(R)$. Furthermore, using the notation in~\ref{parausar2}:
 \begin{enumerate}
    \item The right truncation functor $\tg{}{\phi}$ is computed as
          the composition
          \[
           \tg{ }{\phi_n}{}\tg{ }{\phi_{n-1}}{}\cdots\,\tg{ }{\phi_s}{}.
          \]
    \item For all $X \in \D(R)$, $\tleq{ }{\phi_n}{}\tg{}{\phi'}
          X \in\Dcacb{n}{n}(R)$.
 \end{enumerate}
\end{prop}

\begin{proof}
For simplicity we assume in the proof that $[s,n]$ is the interval $[0,n]$. Let us set $Z_i:=\phi(i)$ for each $i\in[0,n]$. Trivially $\tleq{}{\phi} \Xdot \in\Dleq{n}(R)$. In order to prove the statement $(1)$ we proceed by induction on the length of the sp-filtration $\phi$. The case $\lth(\phi)=n+1=1$ is trivial.
Let $n> 0$, then $1\leq \lth(\phi')<\lth(\phi)$. By inductive hypothesis the result is true for $\phi'.$
Let us consider the commutative diagram of distinguished triangles
  \[
   \begin{tikzpicture}
      \draw[white] (0cm,2cm) -- +(0: \linewidth)
      node (E) [black, pos = 0.45] {$\Xdot$}
      node (H) [black, pos = 0.92] {\mbox{ }}
      node (F) [black, pos = 0.8] {$\tg{ }{\phi_n}{}\tg{}{\phi'} \Xdot$};
      \draw[white] (0cm,1.675cm) -- +(0: \linewidth)
      node (Z) [black, pos = 0.1] {}
      node (L) [black, pos = 0.9] {};
      \draw[white] (0cm,0.7cm) -- +(0: \linewidth)
      node (A) [black, pos = 0.16] {$\tleq{}{\phi'}X$};
      \draw[white] (0cm,3.05cm) -- +(0: \linewidth)
      node (W) [black, pos = 0.1] {}
      node (T) [black, pos = 0.9] {};
      \draw[white] (0cm,3.25cm) -- +(0: \linewidth)
      node (C) [black, pos = 0.37] {$\tleq{ }{\phi_n}{}{\tg{}{\phi'}} \Xdot$};
      \draw[white] (0cm,4.2cm) -- +(0: \linewidth)
      node (U) [black, pos = 0.1] {}
      node (V) [black, pos = 0.9] {};
      \node (B) at (intersection of A--C and F--E) {$N$};
      \node (K) at (intersection of A--E and C--F) {$\tg{ }{\phi'} \Xdot$};
      \node (D) at (intersection of A--B and U--V) {};
      \node (I) at (intersection of A--E and W--T) {};
      \node (J) at (intersection of K--F and Z--L) {};
      \draw [->] (A) -- (B) ;
      \draw [->] (B) -- (C);
      \draw [->] (C) -- (D) node[above, midway, sloped, scale=0.75]{$+$};
      \draw [->] (B) -- (E);
      \draw [->] (E) -- (F) node[below, midway, scale=0.75]{$w$};
      \draw [->] (K) -- (I) node[above, midway, sloped, scale=0.75]{$+$};
      \draw [->] (C) -- (K);
      \draw [->] (K) -- (F) node[above, midway, scale=0.75]{$v$};
      \draw [->] (E) -- (K) node[above, midway, scale=0.75]{$u$};
      \draw [->] (F) -- (J) node[below, midway, sloped, scale=0.75]{$+$};
      \draw [->] (F) -- (H) node[above, midway, sloped, scale=0.75]{$+$};
      \draw [->] (A) -- (E);
   \end{tikzpicture}
  \]
built up from the commutative diagram $w=v {\circ} u$ and the octahedron axiom. The triangle
  \[
   \tleq{}{\phi'}X {\lto}\,\, N \lto 
   \tleq{ }{\phi_n}{}\tg{}{\phi'} X \overset{+}{\lto}
  \]
proves that $\Ndot\in \CU_\phi$ because $\tleq{}{\phi'}X\in \CU_{\phi'}\subset \CU_\phi$ and $\tleq{ }{\phi_n}{}\tg{}{\phi'} X \in\cu{ }{\phi_n}\subset \CU_\phi$. Note that ${\CU_{\phi'}}^{\perp}=\cap_{i=0}^{n-1}{\CU_{\phi_i}}^{\perp}$ and ${\CU_{\phi}}^{\perp}={\CU_{\phi'}}^{\perp}\cap{\CU_{\phi_n}}^{\perp}$, therefore Lemma~\ref{componiendo-2-trucados-derecha} shows that $\tg{ }{\phi_n}{}\tg{}{\phi'} X$ belongs to $\CU_\phi^{\,\,\perp}.$ Hence the horizontal triangle in the above diagram is the $\phi$-triangle with central vertex $X$, that is $\tleq{}{\phi} X=\Ndot$ and $\tg{}{\phi} X=\tg{ }{\phi_n}{}\tg{}{\phi'} X$.

The next question we address is to show that $\tleq{ }{\phi_n}{} \tg{}{{\phi^\prime}} X=\tleq{n}{}\R\Gamma_{\phi(n)} \tg{}{{\phi^\prime}} X$ belongs to $\Dcacb{n}{n}(R)$ or, equivalently, that $\R \Gamma_{Z_n} \tg{}{{\phi^\prime}} X \in\Dg{n-1}(R)$.
To prove it we begin by recalling some useful results. Let
  $
    \Gamma_{{Z_{n-1}}/{Z_n}}\colon  \Mod(R){\to}\Mod(R)
  $
be the functor determined by the short exact sequence of functors in $\Mod(R)$
 \[
  0{\lto}\Gamma_{{Z_n}} {\lto}\,\, \Gamma_{{Z_{n-1}}}  \lto 
  \Gamma_{{Z_{n-1}}/{Z_n}}{\lto}0.
 \]
({\cfr} \cite[variation 2 on p. 219]{RD}). Write $\Gamma_{n-1/n}=\Gamma_{{Z_{n-1}}/{Z_n}}$, $\Gamma_n=\Gamma_{Z_n}$ for each $n\in\ZZ$. Deriving these functors on the right the above abelian exact sequence provide a  distinguished triangle for any $Y\in\D(R)$
 \begin{equation}\label{no-repetir}
   \R\Gamma_n Y {\lto}\,\, \R\Gamma_{n-1}Y \lto 
   \R\Gamma_{n-1/n} Y \overset{+}{\lto}
 \end{equation}
Applying the functor $\R\Gamma_n$ to this triangle we get $\R\Gamma_n\R\Gamma_{n-1/n} Y =0$ since $\R\Gamma_n\R\Gamma_{n-1}Y\to \R\Gamma_n\R\Gamma_n Y = \R\Gamma_n Y$ is an isomorphism ({\cfr} \S~\ref{util}). Furthermore the natural transformation $\R\Gamma_{n-1}\to 1$ induces  morphisms of distinguished triangles
  \[
   \begin{tikzpicture}
      \draw[white] (0cm,0.5cm) -- +(0: \linewidth)
      node (32) [black, pos = 0.15] {$\R\Gamma_{n}\R\Gamma_{n-1} Y$}
      node (33) [black, pos = 0.4] {$\R\Gamma_{n-1}\R\Gamma_{n-1} Y$}
      node (34) [black, pos = 0.7] {$\R\Gamma_{{n-1}/n}\R\Gamma_{n-1} Y$}
      node (35) [black, pos = 0.9] {$ $};
      \draw[white] (0cm,1.75cm) -- +(0: \linewidth)
      node (22) [black, pos = 0.15] {$\R\Gamma_{n} Y$}
      node (23) [black, pos = 0.4] {$ \R\Gamma_{n-1} Y$}
      node (24) [black, pos = 0.7] {$\R\Gamma_{{n-1}/n} Y$}
      node (25) [black, pos = 0.9] {$$};
      \draw[white] (0cm,3cm) -- +(0: \linewidth)
      node (12) [black, pos = 0.15] {$\R\Gamma_{n-1}\R\Gamma_{n} Y$}
      node (13) [black, pos = 0.4] {$\R\Gamma_{n-1}\R\Gamma_{n-1} Y$}
      node (14) [black, pos = 0.7] {$\R\Gamma_{n-1}\R\Gamma_{{n-1}/n} Y$}
      node (15) [black, pos = 0.9] {$ $};
      \draw [->] (12) -- (13) node[above, midway, scale=0.75]{$ $};
      \draw [->] (13) -- (14) node[above, midway, scale=0.75]{$ $};
      \draw [->] (14) -- (15) node[above, midway, scale=0.75]{$+$};
      \draw [->] (22) -- (23) node[above, midway, scale=0.75]{$ $};
      \draw [->] (23) -- (24) node[above, midway, scale=0.75]{$ $};
      \draw [->] (24) -- (25) node[above, midway, scale=0.75]{$+$};
      \draw [->] (32) -- (33) node[above, midway, scale=0.75]{$ $};
      \draw [->] (33) -- (34) node[above, midway, scale=0.75]{$ $};
      \draw [->] (34) -- (35) node[above, midway, scale=0.75]{$+$};
      \draw [->] (12) -- (22) node[right, midway, scale=0.75]{$\wr$};
      \draw [->] (13) -- (23) node[right, midway, scale=0.75]{$\wr$};
      \draw [->] (14) -- (24) node[right, midway, scale=0.75]{};
      \draw [->] (32) -- (22) node[right, midway, scale=0.75]{$\wr$};
      \draw [->] (33) -- (23) node[right, midway, scale=0.75]{$\wr$};
      \draw [->] (34) -- (24) node[right, midway, scale=0.75]{};
   \end{tikzpicture}
  \]
that are in fact isomorphisms of triangles because the two vertical maps on the left are isomorphisms (see  \S \ref{util}). Therefore
  \[
   \R\Gamma_{{{n-1}}}\R\Gamma_{{{n-1}}/{n}} Y
   \cong{\R\Gamma_{{n-1/n}}}\R\Gamma_{{{n-1}}}
   Y\cong{\R\Gamma_{{n-1/n}}} Y
  \]
for all $\Ydot\in\D(R)$.

Going back to our aim, set $\Ydot:=\tg{}{{\phi'}}X$. Note that $\Ydot\in(\cu{ }{{\phi_{n-1}}})^\perp$, so that $\tleq{n-1}{}\R \Gamma_{n-1}\Ydot=0$, that is, $\R\Gamma_{n-1}\Ydot\in\Dg{n-1}(R)$. But then  $\R\Gamma_{n-1/n} \Ydot \cong \R\Gamma_{n-1/n}\R \Gamma_{{{n-1}}} \Ydot$ also belongs to $\Dg{n-1}(R)$. Now  the existence of the triangle (\ref{no-repetir}) allows us to conclude that $\R\Gamma_{{n}}\tg{}{\phi'}X=\R\Gamma_{{n}}\Ydot\in \Dg{n-1}(R)$, as desired.
\end{proof}

\begin{cor}
\label{segundo-cor-lema}
Let us consider the notation in the above proposition. Then for each $X\in\D(R)$ it holds that $\tleq{ }{\phi_n}{}{\tg{}{\phi'}} \Xdot \cong \h^n(\tleq{}{\phi}X)[-n]$ and there is a diagram of distinguished triangles in $\D(R)$
  \[
   \begin{tikzpicture}
      \draw[white] (0cm,2cm) -- +(0: \linewidth)
      node (E) [black, pos = 0.45] {$\Xdot$}
      node (H) [black, pos = 0.92] {\mbox{ }}
      node (F) [black, pos = 0.8] {$\tg{}{\phi} \Xdot$};
      \draw[white] (0cm,1.675cm) -- +(0: \linewidth)
      node (Z) [black, pos = 0.1] {}
      node (L) [black, pos = 0.9] {};
      \draw[white] (0cm,0.7cm) -- +(0: \linewidth)
      node (A) [black, pos = 0.16] {$\tleq{}{\phi'}X$};
      \draw[white] (0cm,3.05cm) -- +(0: \linewidth)
      node (W) [black, pos = 0.1] {}
      node (T) [black, pos = 0.9] {};
      \draw[white] (0cm,3.25cm) -- +(0: \linewidth)
      node (C) [black, pos = 0.37] {$\tleq{ }{\phi_n}{}{\tg{}{\phi'}} \Xdot$};
      \draw[white] (0cm,4.2cm) -- +(0: \linewidth)
      node (U) [black, pos = 0.1] {}
      node (V) [black, pos = 0.9] {};
      \node (B) at (intersection of A--C and F--E) {$\tleq{}{\phi}X$};
      \node (K) at (intersection of A--E and C--F) {$\tg{ }{\phi'} \Xdot$};
      \node (D) at (intersection of A--B and U--V) {};
      \node (I) at (intersection of A--E and W--T) {};
      \node (J) at (intersection of K--F and Z--L) {};
      \draw [->] (A) -- (B) ;
      \draw [->] (B) -- (C);
      \draw [->] (C) -- (D) node[above, midway, sloped, scale=0.75]{$+$};
      \draw [->] (B) -- (E);
      \draw [->] (E) -- (F) node[below, midway, scale=0.75]{$w$};
      \draw [->] (K) -- (I) node[above, midway, sloped, scale=0.75]{$+$};
      \draw [->] (C) -- (K);
      \draw [->] (K) -- (F) node[above, midway, scale=0.75]{$v$};
      \draw [->] (E) -- (K) node[above, midway, scale=0.75]{$u$};
      \draw [->] (F) -- (J) node[below, midway, sloped, scale=0.75]{$+$};
      \draw [->] (F) -- (H) node[above, midway, sloped, scale=0.75]{$+$};
      \draw [->] (A) -- (E);
   \end{tikzpicture}
  \]
in which:
  \begin{enumerate}
     \item the triangle
           $
            \tleq{ }{\phi_n}{}{\tg{}{\phi'}} \Xdot {\lto}\,\,
            \tg{}{\phi'}X
            \lto  \tg{}{\phi} X \overset{+}{\lto}
           $
           is canonically isomorphic to
           $
            \tleq{}{\phi}\tg{}{\phi'} X {\lto}\,\,
            \tg{}{\phi'}X \lto \tg{}{\phi} \tg{}{\phi'} X
            \overset{+}{\lto}; and
           $
     \item the triangle
           $
            \tleq{}{\phi'}X {\lto}\,\, \tleq{}{\phi}X \lto 
            \tleq{}{\phi_n}{}{\tg{}{\phi'}} \Xdot \overset{+}{\lto}
           $
           is canonically isomorphic to
           $
            \tleq{n-1}{}\tleq{}{\phi}X {\lto}\,\, \tleq{}{\phi}X
            \lto  \tg{n-1}{} \tleq{}{\phi}X \overset{+}{\lto}.
           $
  \end{enumerate}
\end{cor}

\begin{proof}
The diagram whose existence we assert is the diagram of distinguished triangles at the beginning of the proof of Proposition~\ref{Kashiwara=ours}. From the very same proof note that $\tleq{ }{\phi_n}{}\tg{}{\phi'} X\in\CU_{\phi_n}\subset \CU_\phi$ hence the triangle
  \[
   \tleq{ }{\phi_n}{}\tg{}{\phi'} X {\lto}\,\, \tg{}{\phi'}X \lto 
   \tg{}{\phi} X \overset{+}{\lto}
  \]
is the ${\phi}$-triangle with central vertex $\tg{}{\phi'} X$, so assertion $(1)$ follows.

We also derive from Proposition~\ref{Kashiwara=ours} that $(2)$ holds true, since $\tleq{}{\phi'}X\in\Dleq{n-1}(R)$  and $\tleq{ }{\phi_n}{}{\tg{}{\phi'}} X\in\Dcacb{n}{n}(R)\subset \Dg{n-1}(R)$. And also as a consequence $\tleq{ }{\phi_n}{}{\tg{}{\phi'}} \Xdot \cong \h^n(\tleq{}{\phi}X)[-n]$.
\end{proof}


\begin{lem}
\label{induction-tool}
Let $\sharp\in\{-,+,\bb,\text{``blank''}\}$ and let $\phi$ be a finite sp-filtration of $\spec(R)$ determined in the interval $[s,n]\subset \ZZ$. The following statements are equivalent:
  \begin{enumerate}
   \item $\CU_\phi\cap\Dc^\sharp(R)$ is an aisle of $\Dc^\sharp(R)$.
   \item $\CU_{\phi '}\cap\Dc^\sharp(R)$ is an aisle of $\Dc^\sharp(R)$
         and $\h^{n}(\R\Gamma_{\phi(n)}M )$ is a finitely
         generated $R$-module, for every
         $M\in\CU_{{\phi'}}^\perp\cap\Dc^\sharp(R)$.
   \item $\CU_{{\phi '}}\cap\Dc^\sharp(R)$ is an aisle of
         $\Dc^\sharp(R)$ and $\tleq{n}{}(\R\Gamma_{\phi(n)}M)
         \in\Dc^\sharp(R)$, for every
         $M\in\CU_{{\phi'}}^\perp\cap\Dc^\sharp(R)$.
 \end{enumerate}
\end{lem}

\begin{proof}
As a consequence of the remark after Corollary~\ref{discreteness-of-filtration2}, it is enough to prove the current Lemma for $\sharp = \text{``blank''}$. 

Take $M\in\CU_{{\phi'}}^\perp\cap\Dc(R)$ so the canonical map $\Mdot\to\tg{}{{\phi'}}\Mdot$ is an isomorphism. Then $\tleq{}{\phi}\Mdot\cong \tleq{ }{\phi_n}{}\Mdot$ by Corollary~\ref{segundo-cor-lema}$(1)$. From the initial statement of that same corollary we get that $\tleq{}{\phi}\Mdot\cong \tleq{ }{\phi_n}{}\Mdot = \tleq{n}{}\R\Gamma_{\phi(n)}\Mdot \cong \h^{n}(\R\Gamma_{\phi(n)}\Mdot)[-n].$ This proves $(2)\dimp (3)$. And also proves that $\tleq{}{\phi}\Mdot\in\Dc(R)$ (equivalently $\tg{}{\phi}\Mdot\in\Dc(R)$) if and only if $\h^{n}(\R\Gamma_{\phi(n)}\Mdot)$ is a finitely generated $R$-module. This said, to prove $(1) \imp (2)$ we just need to check that $\CU_{\phi'}\cap\Dc(R)$ is an aisle of $\Dc(R)$. But it follows from the fact that, for every $\Xdot\in \Dc(R)$, $\tleq{}{\phi'}\Xdot\cong \tleq{n-1}{}{} \tleq{}{\phi}\Xdot$, see Corollary~\ref{segundo-cor-lema}(2).

Finally let us show $(2) \imp (1)$. Let $\Xdot \in \Dc(R)$, assuming $(2)$ we have that both  $\Mdot = \tg{}{{\phi'}}\Xdot$ and $\tleq{ }{\phi_n}{}\tg{}{\phi'} \Xdot = \h^{n}(\R\Gamma_{\phi(n)}M )[-n]$ (see Proposition~\ref{Kashiwara=ours}(2)) belong to $\Dc(R)$ and we conclude by the triangle in Corollary \ref{segundo-cor-lema}$(1)$.


\end{proof}

\begin{prop}
\label{dos-pasos}
Let $\phi$ be a sp-filtration satisfying the  weak Cousin condition  and such that  $\lth(\phi_\ip)\leq 2$ for all prime ideal $\ip\in\spec(R)$ (equivalently for all maximal ideal $\ip\in\spec(R)$), then $\CU_{\phi}\cap\Dc(R)$ is an aisle of $\Dc(R)$.
\end{prop}

\begin{proof}
The question is local so we can assume that $R$ is local (hence $\spec(R)$ is connected and of finite Krull dimension). Then  $\phi$ is a finite  sp-filtration of length $\leq 2$ which, without loss of generality, we assume $\phi$ nonconstant and concentrated in the interval $[0,n]$. If $\lth(\phi)=1$, that is $n=0$, then $\CU_{\phi}\cap\Dc(R)$ is trivially an aisle since in the present setting the condition $\lth(\phi)=1$ is equivalent to saying that $\CU_{\phi}=\Dleq{0}(R)$. If $\lth(\phi)=2$, that is $n=1$, then $\CU_{{\phi'}}\cap\Dc(R)=\Dcleq{0}(R)$ is the aisle of the canonical $t$-structure on $\Dc(R)$. On the other
hand, if $M\in{\CU_{\phi'}}^\perp\cap\Dc(R) =\Dg{0}(R)\cap\Dc(R)$, we have that $\h^1(\R\Gamma_{\phi(1)}M
)\cong\Gamma_{\phi (1)}(\h^1(M ))$, which is finitely
generated because it is a submodule of the finitely generated
module $\h^1(M )$.  Hence $\CU_{{\phi}}\cap\Dc(R)$ is an aisle of $\Dc(R)$ by Lemma~\ref{induction-tool}.
\end{proof}


\section{The classification over rings with dualizing complex}

We begin this section by recalling briefly some basic results on dualizing complexes from \cite[Chapter V \S 2]{RD} in our context. 

\begin{cosa}
A complex $X\in \D(R)$ is \emph{reflexive with respect to} $\DDD\in \D(R)$ if the natural
morphism
  \[
   \sigma_{X}\colon X\lto
   \rhomdot_R(\rhomdot_R(X ,\DDD   ),\DDD   )
  \]
is an isomorphism in $\D(R)$.

Let $\DDD \in \Dbc(R)$ be a complex quasi-isomorphic to a bounded complex of injective $R$-modules. Then following assertions are equivalent \cite[Chapter V \S 2 Proposition~2.1]{RD}:
  \begin{enumerate}
     \item  The contravariant functor
            $\rhomdot_R(-,\DDD)\colon \Dc(R)\to \Dc(R)$ is
            a triangulated duality quasi-inverse of itself.
     \item  The contravariant functor
            $\rhomdot_R(-,\DDD)\colon \Dbc(R)\to \Dbc(R)$ is
            a triangulated duality quasi-inverse of itself.
     \item  Every finitely generated $R$-module is reflexive with respect to $\DDD$.
     \item  The stalk complex
            $R[0]$ is reflexive with respect to $\DDD$.
   \end{enumerate}
   
   A complex $\DDD \in \Dbc(R)$ quasi-isomorphic to a bounded complex of injective $R$-modules that satisfies the above equivalent conditions is called a \emph{dualizing complex} for $R$. More generally, $\DDD \in \Dcmas(R)$ is called  a \emph{pointwise dualizing complex} for $R$ in case $\DDD_\ip$ is a dualizing complex over $R_\ip$, for every $\ip\in \spec(R)$.
   
   If $R$ possesses a dualizing complex then $R$ has finite Krull dimension ({\cfr} \cite[Chapter~V, Corollary~7.2, p.~283]{RD}). Furthermore, $\DDD  $ is a dualizing complex for $R$ if and only if $\DDD  $ is a pointwise dualizing complex and the Krull dimension of $R$ is finite.
\end{cosa}

\begin{cosa}
\label{interpreting-codimension-function}
Let $\DDD  \in \Dbc(R)$ be a complex. As we easily derive from \cite[Chapter~V, Proposition~3.4, p.~269]{RD},  $\DDD  $ is a pointwise dualizing complex if, and only if, for each $\ip\in \spec(R)$ there is a unique $i_\ip\in\ZZ$ such that
  \[
   \Hom_{\D(R_\ip)}(k(\ip),\DDD  _\ip[j])=
   {\begin{cases}
   0,      & \text{if $j\neq i_\ip$},\\
   k(\ip), & \text{if $j=i_\ip$}.
   \end{cases}}
  \]
In that case we define a map $\ddd\colon \spec(R)\to\ZZ$ by setting $\ddd(\ip)=i_\ip$, for all $\ip\in\spec(R)$. Observe that the map $\ddd\colon \spec(R)\to\ZZ$ obeys the rule:
  \begin{center}
     $\ddd(\ip)=i\,\,\Longleftrightarrow\,[\,
     \Hom_{\D(R_\ip)}(k(\ip),\DDD  _\ip[j])=0$, $\forall\,j\in\ZZ$ such that $j\neq i$\,]
  \end{center}
Moreover $\ddd\colon \spec(R)\lto\ZZ$ is a codimension function, that is, if $\ip\subsetneq\iq$ and $\hhtt(\iq/\ip)=1$ then $\ddd(\iq)=\ddd(\ip)+1$ \cite[Chapter~V, \S~7 Proposition~7.1]{RD}.

The following Lemma gives a useful characterization of $\ddd\colon  \spec(R)\to\ZZ$.
\end{cosa}

\begin{lem}
\label{lema}
If $\DDD  \in\D(R)$ is a dualizing complex, then
  \[
   \ddd(\ip)=\max\{n\in\ZZ\,\,;\,\, \R\Gamma_{\V(\ip)} \DDD \in\Dgeq{n}(R)\}
  \]
for each $\ip\in\spec(R)$.
\end{lem}

\begin{proof}
First note that for every $\ip\in\spec(R)$ and $j\in\ZZ$ the support of the $R$-module $\Hom_{\D(R)}(R/\ip,\DDD[j])$ is contained in $\V(\ip)$. Moreover
  \[
   \Hom_{\D(R)}(R/\ip,\DDD[j])_\ip\cong \Hom_{\D(R_\ip)}(k(\ip),\DDD_\ip[j])
  \]
because $\DDD  $ is bounded below.  
The result mentioned in \ref{interpreting-codimension-function} guaranties that $0\neq (\rhomdot_R(R/\ip,\DDD  ))_\ip \in \Dcacb{\ddd(\ip)}{\ddd(\ip)}(R_\ip)$. As a consequence $(\R\Gamma_{\V(\ip)} \DDD)_\ip\cong \R\Gamma_{\ip R_\ip}\DDD_\ip$ belongs to $\Dgeq{\ddd(\ip)}(R_\ip)$ 
and does not belong to $\Dg{\ddd(\ip)}(R_\ip)$.
Whence $\h^{\ddd(\ip)}(\R\Gamma_{\V(\ip)} \DDD)\neq 0$ and therefore
  \[
   \max\{n\in\ZZ\,\,;\,\, \R\Gamma_{\V(\ip)} \DDD \in\Dgeq{n}(R)\}\leq \ddd(\ip),\quad \text{ for all } \ip\in\spec(R).
  \]
Let $T$ be the set of prime ideals in $\spec(R)$ for which the desired equality does not hold, {\ie}
  $
  T=\{\ip\in\spec(R)\,\,;\,\,\R \Gamma_{\V(\ip)} \DDD \notin\Dgeq{\ddd(\ip)}(R)\}.
  $
Assume that $T$ is non empty and choose a prime ideal $\ip\in T$ maximal among the prime ideals in $T$ (recall that $R$ is Noetherian). Then for all $\iq\in\spec(R)$ such that $\ip\subsetneq\iq$ it holds that
  \[
   \ddd(\iq)=\max\{n\in\ZZ\,\,;\,\,\R\Gamma_{\V(\iq)} \DDD \in\Dgeq{n}(R)\}
  \]
Let $W_\ip := \{\iq \,\,;\,\,\iq\in\spec(R)\text{ and }\iq \nsubseteq \ip \} =\spec(R)\setminus \spec(R_\ip)$. Let us consider the canonical Bousfield triangle determined by $W_\ip$ for $\R\Gamma_{\V(\ip)} \DDD$
  \[
   \R\Gamma_{W_\ip}\R\Gamma_{\V(\ip)} \DDD   {\lto}\,\,
   \R\Gamma_{\V(\ip)} \DDD \overset{u}{\lto} 
   (\R\Gamma_{\V(\ip)} \DDD)_\ip \overset{+}{\lto}.
  \]
By the remark in the previous paragraph $(\R\Gamma_{\V(\ip)} \DDD)_\ip$ is in $\Dgeq{\ddd(\ip)}(R)$ and does not belong to $\Dg{\ddd(\ip)}(R)$. Let us prove that the left vertex in the above triangle is in $\Dgeq{\ddd(\ip)+1}(R)$. Let ${W'_\ip}$ be the set of prime ideals ${W'_\ip}:=W_\ip\cap\V(\ip)=\V(\ip)\setminus \{\ip\}$. Then using the canonical isomorphism $ \R\Gamma_{W_\ip}\R\Gamma_{\V(\ip)}\cong\R\Gamma_{W'_\ip}$ ({\cfr} \S~\ref{util}) we deduce that $\R\Gamma_{W_\ip}\R\Gamma_{\V(\ip)} \DDD \cong \R\Gamma_{W'_\ip} \DDD  \in\Dgeq{\ddd(\ip)+1}(R)$, because
  \[
   \R\Gamma_{\V(\iq)} \DDD \in\Dgeq{\ddd(\iq)}(R)\subset \Dgeq{\ddd(\ip)+1}(R),
  \]
for all $\iq\in W'_\ip$ (see Corollary~\ref{corimportante2}). From the above Bousfield triangle we conclude that $\R\Gamma_{\V(\ip)} \DDD \in\Dgeq{\ddd(\ip)}(R)$ against the fact that $\ip\in T$.
\end{proof}

\begin{cosa}
\label{notacionutil}
Let $\DDD  \in\Dbc(R)$ be a dualizing complex for $R$ and $\ddd\colon \spec(R)\to\ZZ$ its associated codimension function.

The duality functor $\rhomdot_R(-,\DDD)\colon \Dbc(R)\lto\Dbc(R)$ transforms the canonical $t$-structure on $\Dbc(R)$ onto a $t$-structure on $\Dbc(R)$. We call this $t$-structure the \emph{Cohen-Macaulay $t$-structure} on $\Dbc(R)$ with respect to $\DDD$, because it can be proved that the objects in its heart are precisely the Cohen-Macaulay complexes in the sense of \cite[pp. 238-239]{RD}.

By Corollary~\ref{islasinducidas} there exists a unique sp-filtration of $\spec(R)$ associated to the Cohen-Macaulay $t$-structure on $\Dbc(R)$ (with respect to $\DDD$). We denote this filtration by
\[
\fcm \colon \ZZ\lto\CP(\spec(R))
\]
and we name it the \emph{Cohen-Macaulay filtration} (with respect to $\DDD$).

Trivially the filtration $\fcm$ satisfies the  weak Cousin condition, actually as a consequence of Proposition~\ref{dual-of-canonical-t-structure} right below, the filtration $\fcm$ does satisfy the strong Cousin condition ({\cfr} the remark after Theorem~\ref{5.5}) because $\ddd$ is a codimension function. 
\end{cosa}

\begin{prop}
\label{dual-of-canonical-t-structure}
Let us consider the hypothesis and notation in the above paragraph. The Cohen-Macaulay filtration
$
\fcm
$
attaches to each $i\in \ZZ$ the set
  \[
   \fcm (i)=\{\ip\in \spec(R)\,\,;\,\,\ddd (\ip)>i\}.
  \]
%
%
\end{prop}

\begin{proof}
Let $(\CV,\CE[1])$ be the $t$-structure on $\Dbc(R)$ image by the duality functor $\rhomdot_{\D(R)}(-,\DDD)$ of the canonical $t$-structure on $\Dbc(R)$. The class $\CV$ consists of those complexes $\Xdot\in\Dbc(R)$ such that
  \[
   0=\Hom_{\D(R)}(\Xdot, \rhomdot_R(\Ndot,\DDD))
  \]
for all $\Ndot\in \Dleq{0}{}(R) \cap \Dbc(R)$. The canonical aisle $\D^{\leq 0}_{\tf}(R)$ is generated by the stalk complex $R$, therefore a complex $\Xdot\in\Dbc(R)$ is in $\CV$ if and only if
  \begin{equation}
  \label{dualdelacanonica}
   0=\Hom_{\D(R)}(\Xdot, \rhomdot_R(R[i],\DDD))\cong
   \Hom_{\D(R)}(\Xdot[i], \DDD)
  \end{equation}
for all $i \geq 0$.
Then the filtration $\fcm$ is defined, for each $i\in\ZZ$, by
  \[
   \fcm (i)=\{\ip\in \spec(R)\,;\,\,
   \Hom_{\D(R)}({R}/{\ip}[j],\DDD  )=0,\text{ for all }j\geq -i\},
  \]
a formula that can be rewritten as 
  \[
   \fcm (i)=\{\ip\in \spec(R)\,\,;\,\,
   \R\Gamma_{\V(\ip)} \DDD \in\Dg{i}{}(R)\}.
  \]
(see Corollary~\ref{ejemplo} and Proposition~\ref{truncacionparaunideal}). Then we get from Lemma~\ref{lema} that $\fcm(i)=\{\ip\in \spec(R)\,\,;\,\,\ddd (\ip)>i\}$, for all $i\in\ZZ$.\qedhere
\end{proof}

\begin{cosa}
\label{islasduales}
Given  a total pre-aisle $\CV$ of $\Dbc(R)$ and $\CE$ its right orthogonal in $\Dbc(R)$ there is a unique sp-filtration $\phi \in \Filsup(R)$ such that $\CV=\CU_\phi\cap \Dbc(R)$ and $\CE=\CF_\phi\cap\Dbc(R)$, where $\CF_\phi$ is the right orthogonal of $\CU_\phi$ in $\D(R)$ (see Corollary~\ref{islasinducidas}). Assume that $R$ admits a dualizing complex $\DDD$ with codimension function $\ddd\colon \spec(R)\to \ZZ$. Then the image by the duality functor $\rhomdot_R(- ,\DDD)$ of the class $\CE$ is a total pre-aisle of $\Dbc(R)$ that we denote by   $\CE^\ddd$. The right orthogonal of $\CE^\ddd$ in $\Dbc(R)$ is the image of $\CV$ by the duality functor $\rhomdot_R(-, \DDD)$, that we denote by $\CV^\ddd$. Therefore there exist a unique sp-filtration $\phi^\ddd \in \Filsup(R)$, that we call the \emph{dual of} $\phi$ (with respect to $\DDD$), such that $\CE^\ddd=\CU_{\phi^\ddd}\cap \Dbc(R)$ and $\CV^\ddd=\CF_{\phi^\ddd}\cap\Dbc(R)$.

Recall that $\Xdot\in \CE=\CF_\phi\cap \Dbc(R)$ if and only if $\Hom_{\D(R)}({R}/{\ip}[-j],\Xdot)=0$, for all $j\in \ZZ$ and $\ip\in \phi(j)$. Then it follows from duality that $\Xdot\in \CE$ if and only if
  \[
   \Hom_{\D(R)}(\rhomdot_R(X ,\DDD),
   \rhomdot_R({R}/{\ip}[-j],\DDD))=0,
  \]
for any $j\in\ZZ$ and $\ip\in \phi(j)$. That is, $\CE^\ddd=\CU_{\phi^\ddd}\cap \Dbc(R)$ is the left orthogonal in $\Dbc(R)$ to the set of objects
$
\CY=\{\rhomdot_R({R}/{\ip}[-j],\DDD)\,;\,\, j\in\ZZ,\,\, \ip\in \phi(j) \}
$.
Therefore the dual of $\phi$ is the sp-filtration defined by
\[
\phi^\ddd(k)=\{\,\iq\in\spec(R)\,;\,\, \Hom_{\D(R)}(R/\iq[-k],\Ydot)=0
 \text{ for all }\Ydot \in\CY\}
 \]
for each $k\in\ZZ$.
\end{cosa}

\begin{lem}
\label{Kashiwara1-generalized}
Let $R$ be a ring that admits a dualizing complex $\DDD$ with  $\fcm\colon \ZZ\to\CP(\spec(R))$ as its associated Cohen-Macaulay filtration, and let $Z\subset \spec(R)$ be a sp-subset. For a complex $X\in \Dbc(R)$ and $n \in \ZZ$, the following assertions are equivalent:
  \begin{enumerate}
     \item $\R\Gamma_Z X$ belongs to $\Dg{n}(R)$;
     \item for each $k\in\ZZ$ and all
           $\iq\in \supp(\Hom_{\D(R)}(X ,\DDD[k]))$, one has that
             \[
              \supp(\tor_i^R({R}/{\iq},{R}/{\ip}))
              \subset\fcm (k+n-i)
              \]
           for all $i\geq 0$ and all $\ip\in Z$;
     \item $Z\cap \supp(\Hom_{\D(R)}(X ,\DDD[k])) \subset\fcm (k+n)$, for all $k\in\ZZ$.
  \end{enumerate}
\end{lem}

\begin{proof}
Let $\phi$ be the sp-filtration determined by the aisle $\CU_Z^n\subset \D(R)$.
Recall from the above paragraph that
  \[
   \phi^\ddd(k)=\big\{\iq\in \spec(R)\,\,\big/
   \,\,{R}/{\iq}[-k]\in{ }^\perp\mathcal{Y}\big\},
  \]
where now
$
   \CY :=\{\rhomdot_R({R}/{\ip}[-j],\DDD)\,\,\big/\,\, \ip\in Z,\, j\leq n \}.
$

Note that $\R\Gamma_Z X\in \Dg{n}(R)$ is equivalent to
\begin{enumerate}
\item[$(1')$]
 \(
   \supp(\h^k(\rhomdot_R(X,\DDD)))=\supp(\Hom_{\D(R)}(X,\DDD  [k]))
   \subset \phi^\ddd(k)
  \)
\end{enumerate}
for all $k\in\ZZ$. In order to prove the equivalence between $(1)$ and $(2)$ we will give an alternative description of the dual sp-filtration $\phi^\ddd$.

Let us fix $k\in\ZZ$ an arbitrary integer. Notice that $\iq\in\phi^\ddd(k)$ if and only if
  \begin{align}
  \label{primeradef}
      0&=\Hom_{\D(R)}({R}/{\iq}[-k],\rhomdot_R({R}/{\ip}[-j],\DDD))\\
      & \cong \Hom_{\D(R)}({R}/{\iq},\rhomdot_R({R}/{\ip},\DDD)[k+j])\notag\\
      & =\h^{k+j}(\rhomdot_R({R}/{\iq},\rhomdot_R({R}/{\ip},\DDD)))\notag
  \end{align}
for all $j\leq n$ and all $\ip\in Z$. Using $\otimes-\hom$ adjunction, the latter fact is equivalent to
  \begin{align*}
      0&=\h^{k+j}(\rhomdot_R({R}/{\iq}\otimes_R^{\LL}{R}/{\ip},\DDD  )\\
       &\cong \h^{j-n}(\rhomdot_R({R}/{\iq}\otimes_R^{\LL}{R}/{\ip}[-k-n],\DDD  ),
  \end{align*}
for all $j\leq n$ and all $\ip\in Z$. Making the change of variables $i=j-n$,  we conclude that $\iq\in\phi^\ddd (k)$ if and only if
  \begin{equation}
  \label{lastcondition1}
   \Hom_{\D(R)}({R}/{\iq}\otimes_R^{\LL}{R}/{\ip}[-k-n],\DDD  [i])=0,
  \end{equation}
for all $i\leq 0$ and all $\ip\in Z$.
Proposition~\ref{4.7}  shows that (\ref{lastcondition1}) is equivalent to saying that
  \begin{align*}
      0&=\Hom_{\D(R)}(\h^s({R}/{\iq}\otimes_R^{\LL}{R}/{\ip}[-k-n])[-s],
      \DDD   [i])\\
      &\cong \Hom_{\D(R)}(\h^{s-k-n}
      ({R}/{\iq}\otimes_R^{\LL}{R}/{\ip})[-s],\DDD   [i])
  \end{align*}
for all $s\in\ZZ$, all $i\leq 0$ and all $\ip\in Z$. The expression labeled (\ref{dualdelacanonica}) in the proof of Proposition~\ref{dual-of-canonical-t-structure} tells us that this last condition for $\iq$ amounts to saying that $\h^{s-k-n}({R}/{\iq}\otimes_R^{\LL}{R}/{\ip})[-s]\in\CU_{\fcm}\cap \Dbc(R)$ or, equivalently, that $\supp(\h^{s-k-n}({R}/{\iq}\otimes_R^{\LL}{R}/{\ip}))\subset\fcm (s)$, for all $s\in\ZZ$ and all $\ip\in Z$. Since the homology of that (derived) tensor product could be nonzero only in case $s-k-n\leq 0$, we make a change of variable $-t=s-k-n$, so that
  \[
   \h^{-t}({R}/{\iq}\otimes_R^{\LL}{R}/{\ip})=
   \tor_t^R({R}/{\iq},{R}/{\ip})
  \]
and $s=n+k-t$, for all $t\geq 0$. We then conclude that $\iq\in\phi^\ddd (k)$ if and only if, for all $t\geq 0$ and all $\ip\in Z$,
  \begin{equation}
  \label{lastcondition2}
  \supp(\tor_t^R({R}/{\iq},{R}/{\ip}))\subset\fcm (k+n-t)
  \end{equation}
Together with the previous paragraph, this characterization of $\phi^\ddd$ proves the looked-for equivalence.

Let us prove now the equivalence between $(2)$ and $(3)$. Let us introduce the sp-filtration ${\xi} \colon \ZZ\to\CP(\spec(R))$ defined by
  \[
   {\xi}(k)=
   \{\iq\in\spec(R)\,\,;\,\,\V(\iq)\cap Z\subset\fcm(k+n)\},
  \]
for each $k\in\ZZ$. An easy exercise shows then that $(3)$ is equivalent to 
  \[
   \supp(\Hom_{\D(R)}(X,\DDD[k]))
   \subset{\xi} (k), \,\,\text{ for all } k\in\ZZ.
  \]
Our goal will be reached once we show that, for each $k\in\ZZ$, ${\xi} (k)$ is the set of prime ideals $\iq\in \spec(R)$ satisfying $\supp(\tor_i^R({R}/{\iq},{R}/{\ip}))\subset\fcm(k+n-i)$ for all $i\geq 0$  and all $\ip\in Z$, otherwise said, ${\xi} =\phi^\ddd$ (see the description of $\phi^\ddd$ in (\ref{lastcondition2})). For that, let us fix an integer $k$. Let us  consider $\iq\in{\xi}(k)$, let $i\geq 0$ be a natural number and take an arbitrary $\ip'\in \supp(\tor_i^R({R}/{\iq},{R}/{\ip}))$. Then
  \[
   \tor_i^{R_{\ip'}}({R_{\ip'}}/{\iq R_{\ip'}},
   {R_{\ip'}}/{\ip R_{\ip'}})\neq 0,
  \]
and this fact implies that $\ip'$ contains both $\iq$ and $\ip$. Then $\ip'\in\V(\iq)\cap\V(\ip)\subset\V(\iq)\cap Z$. Since $\iq\in{\xi} (k)$, we conclude that $\ip'\in\fcm (k+n)$, which implies that $\ip'\in\fcm (k+n-i)$ because $\fcm$ is decreasing. That proves that $\iq\in\phi^\ddd (k)$, so that we get the inclusion ${\xi} (k)\subset\phi^\ddd (k)$. Conversely, assume that $\iq\in\phi^\ddd (k)$ that is (according to (\ref{primeradef}))
  \[
   0=\h^{k+j}(\rhomdot_R({R}/{\iq},
   \rhomdot_R({R}/{\ip},\DDD)))
  \]
for all $j\leq n$ and all $\ip\in Z$ or, equivalently, that
  \begin{equation}
  \label{ecuacion}
      0=\h^{i}(\rhomdot_R({R}/{\iq},\rhomdot_R({R}/{\ip},\DDD))),
  \end{equation}
for any $i\leq k+n$  and all $\ip\in Z$.
We need to prove that $\V(\iq)\cap Z\subset\fcm (k+n)$. Indeed, if $\ip'\in\V(\iq)\cap Z$ then $\ip'\in\phi^\ddd (k)$ because $\phi^\ddd (k)$ is a sp-subset. So the equality in (\ref{ecuacion}) is true for $\iq=\ip=\ip'$, that is
  \begin{align*}
      0=&\h^{i}(\rhomdot_R({R}/{\ip'},\rhomdot_R({R}/{\ip'},\DDD  )))\\
       =&\Hom_{\D(R)}({R}/{\ip'},\rhomdot_R({R}/{\ip'},\DDD  )[i]),
  \end{align*}
for all $i\leq k+n$. But then, viewing $\rhomdot_R({R}/{\ip'},\DDD  )$ as an object of $\D(R/\ip')$, we get that $\Hom_{\D(R/\ip')}({R}/{\ip'},\rhomdot_R({R}/{\ip'},\DDD  )[i])=0$ for all $i\leq k+n$ ({\cfr} \ref{ANTES-0.4}). The last is equivalent to saying that $\rhomdot_R({R}/{\ip'},\DDD  )$ belongs to $\Dg{k+n}(R/\ip')$. But $\rhomdot_R({R}/{\ip'},\DDD  )$ is a dualizing complex over $R/\ip'$ (cf. \cite[Chapter~V, Proposition~2.4, p.~260]{RD}) and then the associated codimension function, which we denote by $\bar{\ddd}$, has the property that $\bar{\ddd}(\bar{0})>k+n$, where $\bar{0}\in\spec(R/\ip')$ is the generic point. It will be enough to check that $\ddd(\ip')\geq\bar{\ddd}(\bar{0})$ or, equivalently, to prove that if $\Hom_{\D(R_{\ip'})}(k(\ip'),\DDD  _{\ip'}[i])\neq 0$ then $\Hom_{\D(k(\ip'))}(k(\ip'),\rhomdot_{k(\ip')}(k(\ip'),{\DDD  }_{\ip'}[i]))\neq 0$. But this last fact follows from remark~\ref{ANTES-0.4}.
\end{proof}

\begin{lem}
\label{Kashiwara2-generalized}
Under the hypothesis of Lemma~\ref{Kashiwara1-generalized} the following assertions are equivalent for $X\in \Dbc(R)$:
  \begin{enumerate}
     \item $\tau^{\leq n} \R\Gamma_Z X $
           belongs to $\Dbc(R);$
     \item for each $k\in\ZZ$ and all
           $\iq\in \supp(\Hom_{\D(R)}(X ,\DDD   [k]))$,
           either  $\iq\in Z$ or $Z\cap\V(\iq)\subset\fcm (k+n).$
  \end{enumerate}
\end{lem}

\begin{proof}
Continuing with the notation in the proof of Lemma \ref{Kashiwara1-generalized}, we have
  \[
   \phi^\ddd (k)={\xi} (k)=\{\iq\in \spec(R)\,\,;\,\,
   \V(\iq)\cap Z\subset\fcm(k+n)\},
  \]
for all $k\in Z$, so that condition $(2)$ can be rewritten as:
  \begin{enumerate}
      \item[$(2')$] $\supp(\Hom_{\D(R)}(X ,\DDD   [k])
      \subset Z\cup\phi^\ddd (k)$, for all $k\in\ZZ$.
  \end{enumerate}

Let us check that $(1)$ implies $(2')$.
Assume that $\tleq{ }{\phi}{}\Xdot=\tleq{n}{}\R\Gamma_Z \Xdot$ belongs to $\Dbc(R)$, then the third vertex in the canonical triangle
  \begin{equation}
  \label{triangulo}
      \tleq{}{\phi}{}\Xdot \lto X\lto
      \tg{}{\phi}{}\Xdot \stackrel{+}{\lto},
  \end{equation}
is in $\Dbc(R)$. Observe that the complex $\R\Gamma_Z\,\tg{ }{\phi}{}\Xdot$ belongs to $\Dg{n}(R)$ since $\tg{}{\phi}{}\Xdot\in{\CU_\phi}^\perp=\cu{n\,\perp}{Z}$. Then, by Lemma~\ref{Kashiwara1-generalized}, we get that
  \[
   Z\cap\supp(\Hom_{\D(R)}(\tg{ }{\phi}{}\Xdot ,\DDD[k]))
   \subset\fcm (k+n),
  \]
for all $k\in\ZZ$ or, equivalently (see the proof of the referred Lemma), that
  \[
   \supp(\Hom_{\D(R)}(\tg{}{\phi}{}\Xdot ,\DDD[k]))
   \subset \phi^\ddd (k),
  \]
for any $k\in\ZZ$. Furthermore, note that
  \[
  (\Hom_{\D(R)}(\tleq{}{\phi}{}\Xdot ,\DDD[k]))_\ip\cong
  \Hom_{\D(R)}((\tleq{ }{\phi}{}\Xdot)_\ip ,\DDD_\ip[k]))
  \]
for any $\ip\in\spec(R)$, therefore $\supp (\Hom_{\D(R)}(\tleq{}{\phi}{}\Xdot ,\DDD[k]))\subset Z$ because $\supp(\tleq{}{\phi}{}\Xdot) \subset Z$. Now  applying the homological functor $\Hom(-,\DDD):=\Hom_{\D(R)}(-,\DDD)$
to the triangle (\ref{triangulo}) we get an exact sequence of $R$-modules
  \[
   \Hom(\tg{ }{\phi}{}\Xdot,\DDD[k])\lto
   \Hom(X ,\DDD[k])\lto
   \Hom(\tleq{ }{\phi}{}\Xdot ,\DDD[k])
  \]
from which $\supp(\Hom_{\D(R)}(X ,\DDD[k]))\subset\phi^\ddd (k)\cup Z$ as desired.

Let us check that $(2')$ implies $(1)$. The functor
\[
\rhomdot_R(-,\DDD)\colon \Dbc(R)\lto \Dbc(R)
\]
is a duality of triangulated categories and the class of objects
  \[
   \CV= \{Y\in \Dbc(R)\,\,;\,\,\supp(\h^k(Y))
   \subset\phi^\ddd (k)\cup Z\text{ for all }k\in\ZZ\}
  \]
can be built up by a finite number of iterated extensions from those objects in the class which are stalk complexes of the form $M[-k]$, with $M$ a finitely generated $R$-module such that $\supp(M)\subset\phi^\ddd (k)\cup Z$. So it is enough to prove that $(2')$ implies $(1)$ for those complexes $\Xdot\in\Dbc(R)$ such that $\rhomdot_R(X,\DDD)\cong M[-k]$, with $\supp(M)\subset\phi^\ddd (k)\cup Z$. Moreover every such $M$ admits a filtration
  \[
   0=M_0\subsetneq M_1\subsetneq \dots
   \subsetneq M_{n-1}\subsetneq M_n=M
  \]
such that ${M_i}/{M_{i-1}} \cong {R}/{\ip_i}$, with $\ip_i\in\supp(M)$ for $i\in\{1,\ldots,n\}$ (see \cite[Theorem 6.4]{ma}). This implies that $M[-k]$ can be built by iterated extensions from stalk complexes $R/\ip [-k]$, with $\ip\in \supp(M)\subset\phi^\ddd (k)\cup Z$. Therefore it is enough to prove $(2')\Longrightarrow (1)$ in the particular case in which $\rhom_R(X ,D )\cong R/\ip [-k]$, with $\ip\in \phi^\ddd (k)\cup Z$.

So let $X=\rhomdot_R({R}/{\ip}[-k],\DDD)$ with $\ip\in \phi^\ddd (k)\cup Z$. If $\ip\in Z$ then $\supp (X )\subset\supp (R/\ip [-k]) =\V (\ip)\subset Z$, whence $\R\Gamma_ZX \cong X$ (by Theorem~\ref{suph}). Then $\tau^{\leq n} \R\Gamma_Z X \cong\tau^{\leq n}X$ belongs to $\Dbc(R)$. In case that $\ip\in\phi^\ddd (k)$, we have $\rhomdot_R(X ,\DDD) \cong {R}/{\ip}[-k]$ so
  \[
   \supp(\Hom_{\D(R)}(X ,\DDD  
   [j]))=\supp(\h^j({R}/{\ip}[-k]))\subset\phi^\ddd (j),
  \]
for all $j\in\ZZ$.  That is exactly what condition $(1')$ in the proof of Lemma~\ref{Kashiwara1-generalized} says, hence $\R\Gamma_Z X\in \Dg{n}R$, that is, $\tau^{\leq n} \R\Gamma_Z X =0$ and assertion $(1)$ trivially holds in this case.
\end{proof}

\begin{thm}
\label{Stanley-conjecture}
Let $R$ be a commutative Noetherian ring that admits a dualizing complex. Given a sp-filtration $\phi \colon \ZZ\to\CP(\spec(R))$ the following are equivalent:
  \begin{enumerate}
     \item $\cu{}{\phi}\cap \Dbc(R)$ is an aisle
           of $\Dbc(R)$;
     \item $\phi$ satisfies the  weak Cousin condition.
  \end{enumerate}
\end{thm}

\begin{proof}
By Corollary~\ref{Cousin} we only need to prove $(2) \imp (1)$. Without loss of generality we may assume that $\spec(R)$ is connected and that $\phi$ is a nonconstant sp-filtration (if it is necessary, localize respect to an idempotent element of $R$ and use Proposition~\ref{localizacion}). Let $\DDD$ be a dualizing complex for $R$, with associated codimension function $\ddd\colon  \spec(R)\to \ZZ$. Under this hypothesis $R$ has finite Krull dimension ({\cfr} \cite[Chapter~V, Corollary~7.2, p.~283]{RD}) and, hence, we know that the sp-filtration $\phi$ is finite of length $\geq 1$. More precisely, there are integers $t\leq n$ such that $\phi$ is determined in the interval $[t,n]\subset \ZZ$, with $\phi(t)=\spec(R)$ (see Corollary~\ref{discreteness-of-filtration2}).

We claim that the following is true for an integer $m\in\ZZ$ and a prime ideal $\iq\in \spec(R)$:
        \begin{quote}
               If $\iq$
               has the property that
               $\V(\iq)\cap\phi (i)\subset\fcm (k+i),$
               for all $i<m$, then either $\iq\in\phi (m)$ or
               $\V(\iq)\cap\phi(m)\subset\fcm (k+m)$.
        \end{quote}
Indeed, suppose that $\iq\not\in\phi(m)$. If $\V(\iq)\cap\phi (m)=\varnothing$ we are done, so we assume that $\V(\iq)\cap\phi (m)$ is nonempty. Choose a minimal element $\ip$ of $\V(\iq)\cap\phi (m)$ and consider a maximal chain of prime ideals
  \[
   \iq=\iq_0\subsetneq\iq_1\subsetneq
   \dots\subsetneq\iq_s=\ip.
  \]
Then the  weak Cousin condition  says that $\iq_{s-1}\in\phi (m-1)$, so that $\iq_{s-1}\in \V(\iq)\cap\phi (m-1)\subset\fcm (k+m-1)$. Therefore $\ddd(\iq_{s-1})>k+m-1$ and, since $\ddd$ is a codimension function, we conclude that $\ddd(\ip)>k+m$ or, equivalently, that $\ip\in\fcm (k+m)$. That proves our claim.

For the rest of the proof assume, without loss of generality, that
  \[
   \spec(R)= \phi (0)\supsetneq \phi (1)\supsetneq
   \dots\supsetneq \phi (n)\supsetneq \phi (n+1)=\varnothing
  \]
with $n+1=\lth(\phi)$. The class $\CU_\phi\cap \Dbc(R)$ is an aisle of $\Dbc(R)$ if $\lth(\phi)=n+1\leq 2$ by Proposition~\ref{dos-pasos}. Suppose that $\lth(\phi)=n+1 > 2$,  then $\phi'$ is a sp-filtration satisfying the  weak Cousin condition  such that $\lth(\phi')=n$. By induction on the length of the sp-filtrations we can assume that $\CU_{\phi'}\cap \Dbc(R)$ is an aisle of $\Dbc(R)$. Then, as a consequence of Lemma~\ref{induction-tool}, checking that $\CU_{\phi}\cap \Dbc(R)$ is an aisle of $\Dbc(R)$  turns out to be equivalent to proving that $\tleq{n}{}\R\Gamma_{\phi (n)}M$ belongs to $\Dbc(R)$  for any $M\in\CU_{\phi'}^\perp\cap \Dbc(R)$. So let $M\in\cu{\perp}{\phi'}\cap \Dbc(R)$, then we have that $\R\Gamma_{\phi (i)}M \in \Dg{i}(R)$, for all $i<n$. From Lemma~\ref{Kashiwara1-generalized}, we conclude that $\phi (i)\cap \supp (\Hom_{\D(R)}(M ,\DDD   [k]))\subset\fcm (k+i)$, for all $k\in\ZZ$ and all $i<n$. This, in particular, implies that if $k\in\ZZ$ and $\iq\in \supp(\Hom_{\D(R)}(M ,\DDD   [k]))$ then $\V(\iq)\cap\phi (i)\subset\fcm (k+i)$, for all $i<n$. Now,  applying the claim above, we get  that, for every $k\in\ZZ$ and every $\iq\in \supp(\Hom_{\D(R)}(M ,\DDD   [k]))$, either $\iq\in\phi (n)$ or $\V(\iq)\cap\phi (n)\subset\fcm (k+n)$. Then, by Lemma~\ref{Kashiwara2-generalized}, we get that $\tau^{\leq n}\R\Gamma_{\phi (n)}M\in \Dbc(R)$ as it is desired.
\end{proof}

\begin{cor}
\label{corStanley-conjecture}
Let $R$ be a commutative Noetherian ring with a pointwise dualizing complex. Then for any sp-filtration $\phi \colon \ZZ\to\CP(\spec(R))$ the following are equivalent:
  \begin{enumerate}
     \item $\cu{}{\phi}\cap \Dc(R)$ is an aisle
           of $\Dc(R)$;
     \item $\phi$ satisfies the  weak Cousin condition.
  \end{enumerate}
\end{cor}

\begin{proof}
Given a complex $X\in \Dc(R)$ let us consider the $\phi$-triangle with central vertex $X$
  \begin{equation}
  \label{triangulo2}
      U\lto X\lto
      V\stackrel{+}{\lto}
  \end{equation}
By localizing at any prime ideal $\ip$, we obtain a $\phi_\ip$-triangle in $\D(R_\ip)$ with central vertex $\Xdot_\ip$
  \begin{equation}
  \label{triangulo2p}
   U_\ip\lto
   X_\ip\lto
   V_\ip
   \stackrel{+}{\lto}.
  \end{equation}
Note that the triangle $(\ref{triangulo2})$ is in $\Dc(R)$ if and only if for all $\ip\in\spec(R)$ the triangle (\ref{triangulo2p}) belongs to $\Dc(R_\ip)$. Furthermore, a sp-filtration $\phi$ of $\spec(R)$ satisfies the weak Cousin condition  if and only for any $\ip\in\spec(R)$ the sp-filtration $\phi_\ip$ of $\spec(R_\ip)$ satisfies the  weak Cousin condition. So we derive the truth of this result from Theorem~\ref{Stanley-conjecture} and Corollary~\ref{5.11} because for each $\ip\in\spec(R)$ the ring $R_\ip$ admits a dualizing complex and, hence, has finite Krull dimension.
\end{proof}

\begin{cor}
\label{conjeturadeStanley}
Let $R$ be a commutative Noetherian ring with dualizing complex. For any $\sharp \in \{-,+,\bb,\text{``blank''}\}$, the assignment $\phi\rightsquigarrow\CU_\phi\cap\Dc^\sharp (R)$ defines a
one-to-one correspondence between:
  \begin{enumerate}
    \item  sp-filtrations of $\spec(R)$
           satisfying the weak Cousin condition;
    \item  aisles of $\Dc^\sharp(R)$ generated by bounded complexes; and
    \item  aisles of $\Dc^\sharp(R)$ generated by perfect complexes.
  \end{enumerate}
In particular, in case $\sharp\in\{-, \bb\}$, the assignment $\phi\rightsquigarrow\CU_\phi\cap\Dc^\sharp (R)$ defines a
bijection between the set of sp-filtrations of $\spec(R)$ and the set of aisles of $\Dc^\sharp (R)$.
\end{cor}

\begin{proof}
Straightforward consequence of Theorem~\ref{Stanley-conjecture}, using Corollary~\ref{5.11}, Theorem~\ref{4.9} and Corollary~\ref{islasinducidas}.
\end{proof}

\end{document}